\documentclass[10pt]{amsart}

\usepackage{amsfonts,amssymb,amsmath,amscd,amstext}
\usepackage[colorlinks=true,linkcolor=blue,citecolor=blue]{hyperref}
\usepackage[utf8]{inputenc}
\usepackage{microtype}
\usepackage{graphicx}
\usepackage{changes}
\usepackage{comment}
\usepackage[a4paper,margin=3.5cm]{geometry}

\renewcommand{\leq}{\leqslant}

\renewcommand{\le}{\leqslant}
\renewcommand{\ge}{\geqslant}
\newcommand{\ptl}{\partial}
\newcommand{\hhh}{{\mathcal{H}}}
\newcommand{\norm}[1]{|| #1 ||}

\newcommand{\rr}{{\mathbb{R}}}

\newcommand{\hh}{{\mathbb{H}}}
\newcommand{\nn}{{\mathbb{N}}}
\newcommand{\sph}{{\mathbb{S}}}

\newcommand{\Om}{\Omega}
\newcommand{\eps}{\varepsilon}

\newcommand{\ga}{\gamma}
\newcommand{\Ga}{\Gamma}

\newcommand{\escpr}[1]{\langle#1\rangle}

\newcommand{\mh}{\mathcal{H}}

\DeclareMathOperator{\divv}{div}
\DeclareMathOperator{\intt}{int}
\DeclareMathOperator{\supp}{supp}

\newtheorem{theorem}{Theorem}[section]
\newtheorem{proposition}[theorem]{Proposition}
\newtheorem{lemma}[theorem]{Lemma}

\theoremstyle{definition}

\newtheorem{remark}[theorem]{Remark}

\newtheorem{definition}[theorem]{Definition} 

\theoremstyle{remark}

\numberwithin{equation}{section}

\definechangesauthor[name=Gianmarco, color=blue]{G}
\definechangesauthor[name=Manuel, color=purple]{M}

\begin{document}

\title[The sub-Finsler Bernstein problem in $\hh^1$]{The Bernstein problem for $(X,Y)$-Lipschitz surfaces in three-dimensional Sub-Finsler Heisenberg groups}

\author[G.~Giovannardi]{Gianmarco Giovannardi}
\address{Dipartimento di Matematica Informatica "U. Dini", Università degli Studi di Firenze, Viale Morgani 67/A, 50134, Firenze, Italy}
\email{gianmarco.giovannardi@unifi.it}

\author[M.~Ritoré]{Manuel Ritoré} 
\address{Departamento de Geometría y Topología \& Research Unit MNat \\
Universidad de Granada \\ Granada \\ Spain}
\email{ritore@ugr.es}

\date{\today}
\thanks{Both authors have been supported by MEC-Feder
	grants MTM2017-84851-C2-1-P and PID2020-118180GB-I00, Junta de Andalucía grants A-FQM-441-UGR18 and P20-00164 and H2020-MSCA-RISE-2017 project GHAIA . The first author has also been supported by INdAM–GNAMPA 2022 Project Analisi geometrica in strutture subriemanniane, codice CUP-E55F22000270001.}


\begin{abstract}
We prove that in the Heisenberg group $\hh^1$ with a sub-Finsler~structure, an $(X,Y)$-Lipschitz surface which is complete, oriented, connected and stable must be a vertical plane. In particular, the result holds for entire intrinsic graphs of Euclidean Lipschitz functions.
\end{abstract}

\maketitle

\thispagestyle{empty}

\section{Introduction}

Variational problems related to the sub-Riemannian perimeter introduced by Capogna, Danielli and Garofalo \cite{MR1312686} (see also Garofalo and Nhieu \cite{MR1404326} and Franchi, Serapioni and Serra Cassano \cite{MR1871966}) have received great interest recently, specially in the Heisenberg groups $\hh^n$. In particular, Bernstein type problems, either for stable intrinsic graphs or for stable surfaces without singular points, have been specially considered. We refer the reader to the introduction in \cite{MR3406514} for an account of recent results, including \cite{MR2165405,MR2435652,MR2405158,MR2333095,MR2648078,MR2609016,MR3044134,MR2875642,MR3259763,MR2448649,MR2481053,MR2262784,MR3984100,MR2583494}. The monograph \cite{MR2312336} provides a quite complete survey of progress on the subject.

In the last years, a left-invariant sub-Finsler perimeter has been considered on the Heisenberg groups, see \cite{snchez2017subfinsler,2020arXiv200704683P,monti-finsler,MR4314055,pozuelo-nilpotent}. A quite natural question is whether Bernstein type results similar to the sub-Riemannian ones hold for the sub-Finsler perimeter.

The main result in this paper is Theorem~\ref{thm:bernstein}, where we prove that in the Heisenberg group $\hh^1$ with a sub-Finsler structure, a complete, oriented, connected and stable $(X,Y)$-Lipschitz surface is a vertical plane. Roughly speaking an $(X,Y)$-Lipschitz surface is locally the intrinsic graph of a Euclidean Lipschitz function. 
Theorem~\ref{thm:bernstein} is a generalization of the corresponding sub-Riemannian result for graphs obtained by Nicolussi and Serra-Cassano in \cite{MR3984100}. Recently, R. Young \cite{MR4433085} proved that a \emph{ruled} area-minimizing entire intrinsic graph in $\mathbb{H}^1$ is a vertical plane by introducing a family of deformations of graphical strips based on variations of a vertical curve.

A sub-Finsler structure is obtained from a left-invariant asymmetric norm $\norm{\cdot}$ in the horizontal distribution of $\hh^1$. Such a norm can be obtained from a convex set $K$ contained in the horizontal plane at the origin in $\hh^1$. The associated $K$-perimeter is defined by
\[
|\ptl E|_K(V)=\sup\bigg\{\int_E\divv(U)\,d\hh^1: U\in\hhh_0^1(V), \norm{U}_{K,\infty}\le 1\bigg\}<+\infty,
\]
where $\hhh_0^1(V)$ is the space of horizontal vector fields of class $C^1$ with compact support in the open set $V$, and $\norm{U}_{K,\infty}=\sup_{p\in V} \norm{U_p}_K$. The integral is computed with respect to the Riemannian measure $d\hh^1$ of a fixed left-invariant Riemannian metric $g$ in $\hh^1$, while the divergence is the one associated to this Riemannian metric. When $K=D$, the closed unit disk centered at the origin of $\rr^2$, the $K$-perimeter coincides with the classical sub-Riemannian perimeter.

The first variation formula of the sub-Finsler perimeter was computed in \cite[\S~3]{2020arXiv200704683P} for surfaces of class $C^2$ without singular points under the hypothesis that $K$ is a convex set of class $C^2_+$. This means that $\ptl K$ is of class $C^2$ and has strictly positive sectional curvature. The following formula was obtained
\begin{equation}
\label{eq:1stintro}
\tag{*}
\dfrac{d}{ds}\Big|_{s=0} A_K( \varphi_s(S))= \int_S u H_K\,dS,
\end{equation}
where $\{\varphi_s\}_{s\in\rr}$ is the flow (i.e., the one-parameter group of diffeomorphisms) associated to a vector field $U$ with compact support in the regular part of $S$, the function $u$ is equal to the normal component $\escpr{U,N}$ of the variation ($N$ is a unit normal for the Riemannian metric $g$), and $dS$ is the Riemannian area element. The function $H_K$ appearing in \eqref{eq:1stintro} is the $K$-mean curvature
\[
H_K=\escpr{D_Z\pi_K(\nu_h),Z},
\]
where $Z$ is a unit horizontal vector field in $S$,  $\nu_h$ is the horizontal unit normal obtained by rotating $Z$ by ninety degrees, and $\pi_K$ is the inverse of the normal map of $\ptl K$. Hence  formula \eqref{eq:1stintro} has sense whenever the horizontal curves in $S$ are of class $C^2$. However, the computations in \cite{2020arXiv200704683P} require to take one derivative of the normal to the surface and so they are not valid for surfaces with less regularity. 

In \cite{MR4314055}, also under the assumption that $K\in C^2_+$, the authors proved that a Euclidean Lipschitz and $\hh$-regular surface with prescribed continuous mean curvature has horizontal (characteristic) curves of class $C^2$, extending the corresponding sub-Riemannian result in \cite{MR3474402}. Hence the $K$-mean curvature can be computed along the characteristic curves in this type of surfaces. Our main task in Section~\ref{sec:first} is to compute the first variation for $(X,Y)$-Lipschitz  surfaces and to check that the first variation formula \eqref{eq:1stintro} also holds for these surfaces with lower regularity. Of course the proof is different from the one in \cite{2020arXiv200704683P} and makes use of a Jacobian of horizontal type introduced by Galli in his Ph.D. Thesis \cite{galli-thesis}, see also \cite{MR3044134}. In particular, for area-stationary surfaces we get $H_K=0$ on $S$. Following the arguments in \cite{MR3474402, MR4314055} we prove that an area-stationary surface $S$ is foliated by horizontal straight lines  and following \cite{MR3984100}   we show that $S$ is $\hh$-regular.

In Section~\ref{sec:codazzi} we show that for an area-stationary surface $S$ the function $y=\escpr{N,T}/|N_h|$ satisfies the differential equation
\[
y''-6y'y+4y^3=0
\]
along almost every horizontal curve in $S$. Here $N$ is a Riemannian unit normal to $S$, $N_h$ the orthogonal projection to the horizontal distribution and $T$ the Reeb vector field on $\hh^1$. The function $D=1/y$ was proven to satisfy the equivalent equation 
\[
DD''=2(D'+1)(D'+2)
\]
for $C^1$ surfaces by Cheng, Hwang, Malchiodi and Yang \cite{MR2983199}. 

Both equations play an important role in the study of the singular set for $C^1$ surfaces. Moreover, the regularity of $\escpr{N,T}/|N_h|$ along the horizontal (characteristic) curves in $S$ is crucial to compute the second variation formula. The function $\escpr{N,T}/|N_h|$ appears frequently in the sub-Riemannian theory of hypersurfaces in the Heisenberg groups $\hh^n$. For instance, it is the curvature of a length-minimizing geodesic realizing the distance between a hypersurface to a given point \cite{MR4193432}.

In Section~\ref{sec:second} we compute the second variation formula of the area for \textit{horizontal} vector fields with compact support. The second variation formula, which is formally similar to the one obtained in the sub-Riemannian case, is given by
\begin{equation*}
\dfrac{d^2}{ds^2}\Big|_{s=0} A_K( \varphi_s(S)) =\int_S \big( Z(u)^2+ q u^2 \big)\, \dfrac{|N_h|}{\kappa(\pi_K({\nu_h}))} dS,
\end{equation*}
where $\{\varphi_s\}_{s\in\rr}$ is the flow associated to a horizontal vector field $U$ with compact support, $u=\escpr{U,N}$ is the normal component and $q$ is the function defined by
\[
\frac{q}{4}=  Z\bigg( \dfrac{\escpr{N,T}}{|N_h|}\bigg)-  \dfrac{\escpr{N,T}^2}{|N_h|^2}.
\]
The function $\kappa$ is the geodesic curvature of $\ptl K$. In the sub-Riemannnian case, where $K$ is the unit disc, we have $\kappa=1$. In our case, the vector $\nu_h$ is constant along horizontal lines in $S$, so that $\kappa(\pi_K(\nu_h))$ is constant on horizontal lines. The computation of this second variation follows the lines of \cite{MR3406514}, where the second variation of the sub-Riemannian area was computed for stable $C^1$ surfaces to solve the Bernstein problem. There is a slight difference in the definition of the function $q$ with respect to \cite{MR3406514} that is related to the choice of $Z$ as $J(\nu_h)$ or $-J(\nu_h)$. We also use some ideas from Nicolussi and Serra-Cassano \cite{MR3984100}, who proved Bernstein's Theorem in the sub-Riemannian setting when $S$ is the intrinsic graph of a Euclidean Lipschitz function.

Finally, in Section~\ref{sec:bernstein} we prove in our main result, Theorem~\ref{thm:bernstein}, that a complete, oriented, connected and stable $(X,Y)$-Lipschitz surface is a vertical plane.  We emphasize that Nicolussi and Serra-Cassano showed that this result is optimal in the sub-Riemannian setting, exhibiting  two counterexamples when the Euclidean Lipschitz regularity assumption is missing, see Theorems~7.1 and 8.1 in \cite{MR3984100}.

\subsubsection*{Acknowledgement}
We warmly thank Francesco Serra Cassano for his advice and for stimulating discussions.

\section{Preliminaries}
\label{sc:preliminaries}
\subsection{The Heisenberg group}
\label{sc:heis}
We denote by $\mathbb{H}^1$ the \emph{first Heisenberg group}, defined as the $3$-dimensional Euclidean space $\rr^3$ with the product
\[
(x,y,t)*(x',y',t')=(x+x', y+y',t+t'+ x'y-xy').
\]
A basis of left invariant vector fields is given by 
\[
X=\dfrac{\partial}{\partial x} + y \dfrac{\partial}{\partial t}, \qquad Y=\dfrac{\partial}{\partial y} - x \dfrac{\partial}{\partial t}, \qquad T=\dfrac{\partial}{\partial t}.
\]
For $p \in \hh^1$, the \emph{left translation by} $p$ is the diffeomorphism $L_p(q) =p*q$.
The \emph{horizontal distribution} $\mathcal{H}$ is the planar distribution generated by $X$ and $Y$, which coincides with the kernel of the contact one-form $\omega=dt-y dx+x dy$. The distribution $\mathcal{H}$ is completely nonintegrable.

We shall consider on $\mathbb{H}^1$ the auxiliary left-invariant Riemannian metric $g= \escpr{\cdot,\cdot}$, so that $\{X, Y, T\}$ is an orthonormal basis at every point. Let $D$ be the Levi-Civita connection associated to the Riemannian metric $g$. 
The following relations can be easily computed 
\begin{equation}
\label{eq:LeviCitiva}
\begin{aligned}
&D_X X=0, \quad & D_Y Y=0, \qquad \, & D_T T=0\\
&D_X Y=-T, \quad & D_X T=Y, \qquad & D_Y T=-X\\
&D_Y X=T, \quad & D_T X=Y, \qquad & D_T Y=-X.\\
\end{aligned}
\end{equation}
Setting $J(U)=D_U T$ for any vector field $U$ in $\hh^1$ we get $J(X)=Y$, $J(Y)=-X$ and $J(T)=0$. Therefore $-J^2$ coincides with the identity when restricted to the horizontal distribution.
 The Riemannian volume of a set $E$ is, up to a constant, the Haar measure of the group and is denoted by $|E|$. The integral of a function $f$ with respect to the Riemannian measure is denoted by $\int f\, d\hh^1$.

\subsection{The pseudo-hermitian connection} The pseudo-hermitian connection  $\nabla$ is the only affine connection satisfying the following properties:
\begin{enumerate}
\item $\nabla$ is a metric connection, and
\item $\text{Tor}(U,V)=2 \escpr{J(U),V} T$ for all vector fields $U,V\in\mathfrak{X}(\hh^1)$.
\end{enumerate}
We recall that a metric connection must satisfy
\[
U(g(V,W))=g(\nabla_U V,W)+g(V,\nabla_U W) 
\]
for vector fields $U,V,W\in\mathfrak{X}(\hh^1)$. The torsion tensor associated to $\nabla$ is defined by
\[
\text{Tor}(U,V) =\nabla_U V - \nabla_V U - [U,V]
\]
for all $U,V\in\mathfrak{X}(\hh^1)$. From this definition and Koszul formula, see formula (9) in the proof of Theorem~3.6 in \cite{MR1138207}, it follows easily that $\nabla X=\nabla Y=0$ and $\nabla J=0$. For a general discussion about  the pseudo-hermitian connection see for instance \cite[§ 1.2]{MR2214654}. Given a curve $\gamma: I \to \hh^1$ we denote by ${\nabla}/{ds}$ the covariant derivative induced by the pseudo-hermitian connection along $\gamma$.

\subsection{Sub-Finsler norms}
Given a convex set $K\subset\rr^2$ with $0\in\intt(K)$ and~associated asymmetric norm $\norm{\cdot}$ in $\rr^2$, we define a left-invariant norm $\norm{\cdot}_K$ on the horizontal distribution of $\hh^1$ by means of the equality
\[
(\norm{fX+gY}_K)(p)=\norm{(f(p),g(p))},
\] 
for any $p\in\hh^1$. The dual norm is denoted by $\norm{\cdot}_{K,*}$. 

If  the boundary of $K$ is of class $C^\ell$, $\ell\ge 2$, and the geodesic curvature of $\ptl K$ is strictly positive, we say that $K$ is of class $C^\ell_+$. When $K$ is of class $C^2_+$, the outer Gauss map $N_K$ is a diffeomorphism from $\ptl K$ to $\sph^1$ and the map
\[
\pi_K(fX+gY)=N_K^{-1}\bigg(\frac{(f,g)}{\sqrt{f^2+g^2}}\bigg),
\]
defined for nowhere vanishing horizontal vector fields $U=fX+gY$, satisfies
\[
\norm{U}_{K,*}=\escpr{U,\pi_K(U)}.
\]
See \S~2.3 in \cite{2020arXiv200704683P}.

\subsection{Sub-Finsler perimeter}
\label{sc:subper}
Here we summarize some of the results contained in subsection 2.4 in \cite{2020arXiv200704683P}.

Given a  compact convex set $K\subset\rr^2$ with $0\in\intt(K)$, the norm $\norm{\cdot}_K$ defines a perimeter functional: given a measurable set $E\subset\hh^1$ and  an open subset $\Om\subset\hh^1$, we say that $E$ has locally finite $K$-perimeter in $\Om$ if for any relatively compact open set $V\subset\Om$ we have
\[
|\ptl E|_K(V)=\sup\bigg\{\int_E\divv(U)\,d\hh^1: U\in\hhh_0^1(V), \norm{U}_{K,\infty}\le 1\bigg\}<+\infty,
\]
where $\hhh_0^1(V)$ is the space of horizontal vector fields of class $C^1$ with compact support in $V$, and $\norm{U}_{K,\infty}=\sup_{p\in V} \norm{U_p}_K$. The integral is computed with respect to the Riemannian measure $d\hh^1$ of the left-invariant Riemannian metric $g$, and the divergence is the one associated to $g$. When $K=D$, the closed unit disk centered at the origin of $\rr^2$, the $K$-perimeter coincides with classical sub-Riemannian perimeter. 

If $K,K'$ are bounded convex bodies containing $0$ in its interior then there exist~constants $\alpha,\beta>0$ such that
\[
\alpha \norm{x}_{K'}\le \norm{x}_K\le \beta \norm{x}_{K'},\quad \text{for all }x\in\rr^2,
\]
and it is not difficult to prove that
\begin{equation*}
\beta^{-1} |\ptl E|_{K'}(V)\leq |\ptl E|_K(V)\leq \alpha^{-1} |\ptl E|_{K'}(V).
\end{equation*}
Then $E$ has locally finite $K$-perimeter if and only if it has locally finite $K'$-perimeter. In particular, any set with locally finite $K$-perimeter has locally finite sub-Riemannian perimeter.

Riesz Representation Theorem implies the existence of a $|\ptl E|_K$-measurable vector field $\nu_K$ so that for any horizontal vector field $U$ with compact support of class $C^1$ we have
\[
\int_\Om\divv(U)\,d\hh^1=\int_\Om\escpr{U,\nu_K}\,d|\ptl E|_K.
\]
In addition, $\nu_K$ satisfies $|\ptl E|_K$-a.e. the equality $\norm{\nu_K}_{K,*}=1$, where $\norm{\cdot}_{K,*}$ is the dual norm of $\norm{\cdot}_K$.  

Given two convex sets $K,K'\subset \rr^2$ containing $0$ in their interiors, we have the following representation formula for the sub-Finsler perimeter measure $|\partial E|_K$ and the vector field $\nu_K$
\begin{equation*}
|\partial E|_K=\norm{\nu_{K'}}_{K,*}|\partial E|_{K'},\quad \nu_{K}=\frac{\nu_{K'}}{\norm{\nu_{K'}}_{K,*}}.
\end{equation*}
Indeed, for the closed unit disk $D\subset\rr^2$ centered at $0$ we know that in the Euclidean Lipschitz case $\nu_D=\nu_h$ and $|N_h|=\norm{N_h}_{D,*}$ where $N$ is the \emph{outer} unit normal. Hence we have
\begin{equation*}
|\partial E|_K=\norm{\nu_h}_{K,*}d|\partial E|_D, \quad \nu_K=\frac{\nu_h}{\norm{\nu_h}_{K,*}}.
\end{equation*}
Here $|\partial E|_D$ is the standard sub-Riemannian measure. Moreover, $\nu_h=N_h/|N_h|$ and $|N_h|^{-1}d|\ptl E|_D=dS$, where $dS$ is the standard Riemannian measure on $S$. Hence we get, for a set $E$ with Euclidean Lipschitz boundary $S$
\begin{equation}
\label{eq:AKlipschitz}
|\ptl E|_K(\Om)=\int_{S\cap\Om}\norm{N_h}_{K,*}\,dS,
\end{equation}
where $dS$ is the Riemannian measure on $S$, obtained from the area formula using a local Lipschitz parameterization of $S$, see Proposition~2.14 in \cite{MR1871966}. It coincides with the $2$-dimensional Hausdorff measure associated to the Riemannian distance induced by $g$. We stress that here $N$ is the \emph{outer} unit normal. This choice is important because of the lack of symmetry of $\norm{\cdot}_K$ and $\norm{\cdot}_{K,*}$. 
Moreover when $S=\ptl E\cap\Omega$ is a Euclidean Lipschitz surface the $K$-perimeter coincides with the area functional 
\[
A_K(S)=\int_{S}\norm{N_h}_{K,*}\,dS.
\]

\subsection{Surfaces in $\hh^1$}
\label{sc:surfaceinH}

Following \cite{MR2223801,MR1871966} we provide the following definition.

\begin{definition}[$\mathbb{H}$-regular surface]
A real  continuous function $f$ defined on an open set $\Omega\subset \mathbb{H}^1$  is of class $C_{\mathbb{H}}^1(\Omega)$ if the distributional derivative $\nabla_{\mathbb{H}} f= (X f, Y f)$ is represented by a continuous vector field on $\Om$.

We say that $S \subset \mathbb{H}^1$ is an $\mathbb{H}$-regular surface if for each $p \in \mathbb{H}^1$ there exist an open set $U$ containing $p$ and a function $f \in C_{\mathbb{H}}^1(U)$ such that $\nabla_{\mathbb{H}} f \ne 0$ on $U$ and $S \cap U=\{f=0\}$.
Under such conditions, the horizontal unit normal $\nu_h$ on $S\cap U$ is defined as the restriction of the non-vanishing continuous vector field
\[
\dfrac{\nabla_{\mathbb{H}} f}{|\nabla_{\mathbb{H}} f|},
\]
defined on all of $U$.
\end{definition}

Given a vertical plane $P\subset\hh^1$, and a function $u$ defined on a domain $D\subset P$, we denote by $\text{Gr}(u)$ the \textit{intrinsic graph} of $u$, defined as the Riemannian normal graph of the function $u$. Since the Riemannian unit normal to $P$ is the restriction of a unitary left-invariant vector field $X_P$,  the intrinsic graph of $u$ is given by
\[
\text{Gr}(u)=\big\{\exp_p\big(u(p)\,X_P(p)\big): p\in D\big\}.
\]
where $\exp$ is the exponential map on the Riemannian manifold $(\hh^1,g)$. Using Euclidean rotations about the vertical axis $x=y=0$, that are isometries of the Riemannian metric $g$, we may assume that $P$ is the plane $\{y=0\}$. Since in this case $X_P=Y$, the intrinsic graph $\text{Gr}(u)$ can be parameterized by the map
\[
\Phi^u (x,t)=(x, u(x,t), t- x u(x,t)),
\]
for $(x,0,t) \in D$. Notice that $\Phi^u (x,t)=(x,0,t)*(0,u(x,t),0)$, where $*$ is the Heisenberg product defined in \ref{sc:heis}. For further details, we refer the reader to \cite{MR2313532}. Note also that $u$ measures the signed distance of $\Phi^u(x,t)$ to the plane $P$, see \cite{MR4193432}.

Given the intrinsic graph $\text{Gr}(u)$ of a Euclidean Lipschitz function defined on some domain $D$ of the vertical plane $P$, we know by Rademacher's Theorem that $u$ is $\mathcal{H}^2$-a.e. differentiable on $D$, where $\mathcal{H}^2$ is the $2$-dimensional Euclidean Hausdorff measure on $D$. Assuming $P=\{y=0\}$, and given a differentiability point $(x_0,0,t_0)$ of $u$, the tangent plane of $\text{Gr}(u)$ is well defined at $\Phi^u(x_0,t_0)$ and so it is the normal vector field $N$. Hence $N$ is defined $\mathcal{H}^2$-a.e. on $\text{Gr}(u)$. Moreover,
\begin{equation}
\label{eq:unitnormal}
 N=\frac{(u_x+2uu_t)X-Y+u_tT}{\sqrt{1+u_t^2+(u_x+2uu_t)^2}},   
\end{equation}
see the computations in \S~4 in \cite{MR3412382}. Hence $N$ is never vertical. At differentiability points of $\text{Gr}(u)$ we define
\[
\nu_h=\frac{N_h}{|N_h|}=\frac{(u_x+2uu_t)X-Y}{\sqrt{1+(u_x+2uu_t)^2}},
\]
and the vector field $Z$ by
\[
Z=-J(\nu_h),
\]
which is tangent to $S$ and horizontal. An orthonormal basis at the tangent space of $\text{Gr}(u)$ at the differentiable point is obtained by adding to $Z$ the vector
\begin{equation}
\label{eq:defE}
E=\escpr{N,T} \nu_h- |N_h| T.
\end{equation}

Following  \cite{MR3060706} we provide the following definition.
\begin{definition}
\label{def:Vittone}
A set $S \subset \hh^1$ is a $(X,Y)$-Lipschitz surface if for each $p \in S$ there exist an open neighborhood $U_p\subset \hh^1$ of $p$, and a Lipschitz function $f:U \to \rr$ such that
\[
S\cap U=\{f=0\}
\]
and 
\[ 
X f\ge l \quad \text{a.e. on }\, U\qquad \text{or} \qquad Yf\ge l \quad \text{a.e. on }\, U\
\]
for a suitable $l>0$.
\end{definition}

We stress that  following result was previously obtained by \cite[Theorem 3.2]{MR3060706}, using the notion of homogeneous cone. Here we provide a different proof. 
\begin{theorem}
\label{th:iftLip}
A set $S\subset \hh^1$ is a $(X,Y)$-Lipschitz surface if and only if $S$ is locally  the intrinsic graph of a Euclidean Lipschitz function.
\end{theorem}
\begin{proof}
Assume that $S$ is a $(X,Y)$-Lipschitz surface. Given $p \in S$ there exist an open ball $B_r(p)\subset\hh^1$ and a Euclidean Lipschitz function $f$ defined on $B_r(p)$ such that 
\[
S\cap B_r(p)=\big\{(x,y,t) : f(x,y,t)=0\big\}.
\]
Since $S$ is $(X,Y)$-Lipschitz, after a rotation about the vertical axis we may assume the existence of $l>0$ such that $Yf(q) \ge l>0$ for every point of differentiability of $f$ close enough to $p$. In particular the convex hull of 
\[
\bigg\{\lim_{i \to\infty }Yf (q_i)   :    \lim_{i\to \infty} q_i = p, q_i \text{ differentiability point of }f \bigg\}
\]
does not contain $0$. Let us consider the $C^{\infty}$ diffeomorphism $H(x,y,t)=(x,y,t-xy)$ on $\hh^1$. Then the function $f \circ H$ is Lipschitz and 
\[
\dfrac{\partial (f \circ H)}{\partial y} (q)= \left(\dfrac{\partial f}{\partial y}- x \dfrac{\partial f}{\partial t}\right) (q)= Yf(q)
\]
for each point $q$ of differentiability of $f$. Therefore by the Implicit function Theorem for Lipschitz functions \cite[p.~255]{MR1019086} there exists an open neighborhood $D\subset \{y=0\}$ of the projection of $p$ on $\{y=0\}$ and a Euclidean Lipschitz function $u: D \to  \rr$ such that  $f(x,u(x,t),t-xu(x,t))=0$. In other words, the surface $S$ is locally an intrinsic graph of a Lipschitz function over the vertical plane $\{y=0\}$.

Assume now that $S$ is locally the intrinsic graph of a Euclidean Lipschitz function $u$. Let $p$ in $S$ and assume that  $S\cap B_r(p)=\Phi^u(D)$ where  $\Phi^u(x,t)=(x,u(x,t),t-xu(x,t))$  and  $u: D \to \rr$ is a Euclidean Lipschitz function.
Setting
 \[
 f(x,y,t)=y-u(x,t+yx),
 \]
we clearly have that $f$ is a Euclidean Lipschitz function defined in an open neighborhood of $p$. Eventually reducing the radius $r>0$ we get that $S\cap B_r(p)=\{f=0\}$ and $Y(f)=1>0$ a.e. on $B_r(p)$. Therefore $S$ is a $(X,Y)$-Lipschitz surface.
\end{proof}

\begin{remark}
Notice that
\begin{enumerate}
    \item An $(X,Y)$-Lipschitz surface is an embedded surface by Theorem \ref{th:iftLip}.
    \item If a Euclidean Lipschitz function $f$ defined on an open domain of $\hh^1$ is $C_\hh^1$ then their level sets are $(X,Y)$-Lipschitz. Indeed, since the horizontal gradient $\nabla_{\hh} f=(Xf,Yf)$ is a never vanishing \emph{continuous} vector field, we obtain that locally $X f\ge l>0$ or $Y(f)\ge l>0$ eventually replacing $f$ by $-f$. 
\end{enumerate}
\end{remark}

\begin{definition}
Let $S \subset \hh^1$ be a $C^1$ surface. We say that $p\in S$ belongs to the \textit{singular set} $S_0$  of $S$ if the tangent space $T_pS$ coincides with the horizontal distribution $\mh_p$.
\end{definition}

The following result, whose proof can be found in \cite{MR3044134}, will be used to compute the first and second variation of a surface.

\begin{proposition}
\label{pr:divt}
Let $S$ be an oriented immersed $C^2$ surface in $\hh^1$ with singular set $S_0=\emptyset$ and let $f\in C^1(S)$. Then
\[
\divv_S( f  Z)=Z(f)- (\escpr{N,T} \theta(E)+2\escpr{N,T} |N_h|)f 
\]
and 
\[
\divv_S( f E)=E(f) + \escpr{N,T} \theta(Z) f ,
\]
where we have set  $\theta(W)=\escpr{\nabla_W \nu_h, Z}$ for each vector field $W$.
\end{proposition}

\section{The first variation formula}
\label{sec:first}

We start this section computing the first variation formula for Lipschitz surfaces which are twice differentiable in the horizontal directions. We start by proving some technical lemmas.

\begin{lemma}
\label{lm:cdlb}
Let $U$ be a smooth vector field on $\hh^1$ and $\{\varphi_s\}_{s\in\rr}$ be the flow associated to $U$. Let $p \in \hh^1$ and $e \in T_p \hh^1$. Define the smooth curve $\beta(s)=\varphi_s(p)$ and the smooth vector field $E(s)=(d\varphi_s)_p(e)$ along $\beta$. Then we have 
\begin{equation}
\label{eq:1stnabla}
\dfrac{\nabla}{ds}\Big|_{s=0} E(s)= \nabla_e U+2\escpr{J(U_p),e}\,T_p.
\end{equation}
\begin{proof}
Let us rename the standard coordinates $(x,y,t)$ as $(x_1,x_2,x_3)$. Let $\varphi_s=(\varphi_1,\varphi_2,\varphi_3)$, and $e=(e_1,e_2,e_3)$, and  $U=\sum_{i=1}^3 f_i \tfrac{\partial}{\partial{x}_i}$. Then
\[
(d \varphi_s)_p=\begin{pmatrix}
\frac{\partial \varphi_1}{\partial x_1} & \frac{\partial \varphi_1}{\partial x_2}& \frac{\partial \varphi_1}{\partial x_3} \\
\frac{\partial \varphi_2}{\partial x_1} & \frac{\partial \varphi_2}{\partial x_2}& \frac{\partial \varphi_2}{\partial x_3}   \\
\frac{\partial \varphi_3}{\partial x_1} & \frac{\partial \varphi_3}{\partial x_2}& \frac{\partial \varphi_3}{\partial x_3} 
\end{pmatrix},
\]
and 
\[
E(s)=(d \varphi_s)_p (e)= \sum_{i=1}^3 \sum_{j=1}^3 e_j  \dfrac{\partial \varphi_i}{\partial x_j}\left(\dfrac{\partial }{\partial x_i}\right)_{\beta(s)}= \sum_{i=1}^3 g_i(s) \left(\dfrac{\partial }{\partial x_i}\right)_{\beta(s)},
\]
where $g_i(s)= \sum_{j=1}^3 e_j  \tfrac{\partial \varphi_i}{\partial x_j}$. Therefore
\[
\dfrac{\nabla}{ds}\Big|_{s=0} E(s)= \sum_{i=1}^3 g_i'(0) \left(\dfrac{\partial }{\partial x_i}\right)_p+ \sum_{i=1}^3 g_i(0)  \nabla_{U_p} \dfrac{\partial }{\partial x_i}.
\]
Since $g_i(0)=e_i$ and $g_i'(0)=e(f_i)$ we have 
\[
\dfrac{\nabla}{ds}\Big|_{s=0} E(s)=\sum_{i=1}^3 e(f_i) \left(\dfrac{\partial }{\partial x_i}\right)_p+ e_i  \nabla_{U_p} \dfrac{\partial }{\partial x_i}.
\]
On the other hand
\[
\nabla_e U= \sum_{i=1}^3 e(f_i) \left(\dfrac{\partial }{\partial x_i}\right)_p+ f_i \nabla_{e} \dfrac{\partial }{\partial x_i}
\]
Since 
\begin{align*}
\nabla_{\tfrac{\partial}{\partial{x}}} \tfrac{\partial}{\partial y}&=\nabla_{\tfrac{\partial}{\partial{y}}} \tfrac{\partial}{\partial x}+ \text{Tor}(\tfrac{\partial}{\partial{x}}, \tfrac{\partial}{\partial{y}})= \nabla_{\tfrac{\partial}{\partial{y}}} \tfrac{\partial}{\partial x}+ 2\,\escpr{J(\tfrac{\partial}{\partial{x}}), \tfrac{\partial}{\partial{y}}}\,T
\\
&= \nabla_{\tfrac{\partial}{\partial{y}}} \tfrac{\partial}{\partial x}+ 2\,\escpr{J(X)-y J(T), Y}\,T+ 2x \,\escpr{J(X), T}\,T
\\
&=\nabla_{\tfrac{\partial}{\partial{y}}} \tfrac{\partial}{\partial x}+ 2T,
\end{align*}
and
$$\nabla_{\tfrac{\partial}{\partial{x_i}}} \tfrac{\partial}{\partial x_j}=\nabla_{\tfrac{\partial}{\partial{x_j}}} \tfrac{\partial}{\partial x_i},$$
for $\{i,j\} \ne \{1,2\}$ there follows, evaluating at $p$,
\begin{equation*}
\begin{aligned}
\sum_{i=1}^3 e_i \nabla_{U_p} \tfrac{\partial }{\partial x_i}&=\sum_{i,j=1}^3  e_i \, f_j\, \nabla_{\tfrac{\partial}{\partial{x_j}}} \tfrac{\partial}{\partial x_i}=e_1 f_2  \nabla_{\tfrac{\partial}{\partial{y}}} \tfrac{\partial}{\partial x} +e_2 f_1 \nabla_{\tfrac{\partial}{\partial{x}}} \tfrac{\partial}{\partial y}  
\\
&\hspace{0.3\textwidth}+\sum_{i,j=1, \{i,j\}\ne\{1,2\}}^3 e_i \, f_j\,\nabla_{\tfrac{\partial}{\partial{x_j}}} \tfrac{\partial}{\partial x_i} \\
&=e_1 f_2 (\nabla_{\tfrac{\partial}{\partial{x}}} \tfrac{\partial}{\partial y}-2T) +e_2 f_1 (\nabla_{\tfrac{\partial}{\partial{y}}} \tfrac{\partial}{\partial x}+ 2T)
\\
&\hspace{0.3\textwidth}+\sum_{i,j=1, \{i,j\}\ne\{1,2\}}^3 e_i \, f_j\, \nabla_{\tfrac{\partial}{\partial{x_i}}} \tfrac{\partial}{\partial x_j} \\
&=2(e_2 f_1- e_1 f_2)\,T+ \sum_{j=1}^3  \, f_j \, \nabla_{e}  \tfrac{\partial}{\partial x_j} 
\\
&=2(e_2 f_1- e_1 f_2)\,T+ \sum_{i=1}^3  \, f_i \,\nabla_{e}  \tfrac{\partial}{\partial x_i}.
\end{aligned}
\end{equation*}
Hence we obtain \eqref{eq:1stnabla}. 
\end{proof}
\end{lemma}

Let $S\subset\hh^1$ be a $(X,Y)$-Lipschitz surface and $p\in S$ a point of differentiability of $S$. Consider a $C^1$ vector field $U$ with compact support on $\hh^1$ and let $\{\varphi_t\}_{t\in\rr}$ be the associated flow. A basis of the tangent space to  $\varphi_s(S)$ at $\varphi_s(p)$ is given~by $E_1(s)=(d\varphi_s)_p(Z_p)$ and $E_2(s)=(d\varphi_s)_p(E_p)$, where $E$ is defined in \eqref{eq:defE}. Let $N_s$ be the unit normal to $\varphi_s(S)$. Then 
\[
[(N_s)_h]_{\varphi_s(p)}=\dfrac{[E_1(s) \times E_2(s)]_h}{|E_1(s) \times E_2(s)|}.
\]
Observe that $\text{Jac}(\varphi_s)(p)=|E_1(s) \times E_2(s)|$.
Setting $V(s,p)=[E_1(s) \times E_2(s)]_h$ we have 
\begin{equation}
    \label{eq:defV}
V(s,p)=\escpr{E_1,T}\, T \times E_2 + \escpr{E_2,T}\, E_1 \times T.
\end{equation}
Since $\{Z, \nu_h,T\}$ is positively oriented, we have $V(0,\cdot)=|N_h|\, \nu_h$.
The area functional $A_K$ of $\varphi_s(S)$ is then given by
\begin{equation}
\label{eq:areaK-V}
\begin{split}
A_K( \varphi_s(S))&= \int_S \escpr{\pi_K([E_1(s) \times E_2(s)]_h), \dfrac{[E_1(s) \times E_2(s)]_h}{|E_1(s) \times E_2(s)|}} |E_1(s) \times E_2(s)|\,d S\\
&=\int_S \escpr{\pi_K(V(s,p)), V(s,p)}\,dS(p)
\\
&=\int_S ||V(s,p)||_{*}\,dS(p).
\end{split}
\end{equation}

\begin{lemma}
\label{lm:novanishing}
Let $S$ be an $(X,Y)$-Lipschitz surface, and $\{\varphi_s\}_{s\in\rr}$ the flow associated to a compactly supported $C^1$ vector field $U$ in $\hh^1$. Let $L=\textit{supp}(U) \cap S$. Then there exist positive constants $l, s_0$ such that  $|V(s,p)|\ge  \frac{l}{2}$ for a.e. $p\in L$ and any $s \in (-s_0,s_0)$. 
\end{lemma}
\begin{proof}
Theorem \ref{th:iftLip} implies that, after a rotation about the vertical axis, there exists $r>0$ and a Euclidean Lipschitz function $u:D \to \rr$, defined on an open set $D$ of the vertical plane $y=0$, such that $S \cap B_r(p)=\Phi^u(D)$. Let $B\subset D$ be the set of  points of differentiability of $u$. Since $u$ is Euclidean Lipschitz, equation \eqref{eq:unitnormal} implies
\[
|N_h|\ge \dfrac{1}{\sqrt{1+(u_t)^2+(u_x+2uu_t)^2}} \ge l_{B_r(p)} >0
\]
for some constant $l_{B_r(p)}$ depending on a Lipschitz constant of $u$ on $B_r(p)$. Therefore by the compactness of $\text{supp}(U)$, there exists $l>0$ such that $|N_h|\ge l >0$ for a.e. $q \in L$.

Assume by contradiction the existence of a mesurable set $A \subset L$ such that $\mathcal{H}^2(A)>0$ and of a sequence $\{s_j\}_{j\in\nn}$ converging to $0$ such that $|V(s_j,q)|< \frac{l}{2}$ for each $q \in A$.  Then
\begin{equation}
    \label{eq:lemma-3.2-1}
\lim_{j \to \infty}\int_A |V(s_j,q)|\, dS(q)\le \frac{l}{2}\,\mathcal{H}^2(A).
\end{equation}
On the other hand, since $|E_1(s_j)|=|d\varphi_{s_j}(Z)| \le C' $, $|E_2(s_j)|=|d\varphi_{s_j}(E)| \le C'$ for some $C'>0$, we obtain from \eqref{eq:defV} the existence of $C>0$ such that $|V(s_j)|\le C$ on $A$. By the continuity of $V(s,q)$ in $s$ we have $\lim_{j\to \infty} |V(s_j)|= |N_h|$ a.e. in $A$. By dominated convergence
\[
\lim_{j \to \infty}\int_A |V(s_j,q)|\, dS(q)= \int_A |N_h|\, dS \ge l\, \mathcal{H}^2(A).
\]
Therefore, since $\mathcal{H}^2(A)>0$ we get a contradiction to \eqref{eq:lemma-3.2-1}
\end{proof}

Now we compute the first variation of the sub-Finsler area.

\begin{proposition}
\label{prop:fvf}
Let $K\in C^2_+$ be a convex body with $0 \in \intt(K)$, and $S \subset \hh^1$  an oriented $(X,Y)$-Lipschitz surface. Then the first variation of the sub-Finsler area induced by a $C^1$ vector field U with compact support in $\hh^1$, with $\partial S \cap \text{supp}(U)= \emptyset$, is given by
\begin{equation}
\label{eq:first0}
\begin{aligned}
\dfrac{d}{ds}\Big|_{s=0} A_K( \varphi_s(S))= \int_S\bigg[&-\escpr{N,T} Z(\escpr{U,T}) \pi_Z -E(\escpr{U,T}) \pi_{\nu}\\
&-  2\escpr{N,T} \escpr{J(U),\pi(\nu_h)} - |N_h| \escpr{J(\pi(\nu_h)), \nabla_Z U}\bigg] dS,
\end{aligned}
\end{equation}
where $\{\varphi_s\}_{s\in\rr}$ is the flow associated to $U$.

Moreover, if we assume that the derivative in the $Z$-direction of $\nu_h$ and $Z$ exists and is continuous, then we have
\begin{equation}
\label{eq:first}
\dfrac{d}{ds}\Big|_{s=0} A_K( \varphi_s(S))= \int_S u\, \escpr{\nabla_Z \pi_K(\nu_h),Z}dS,
\end{equation}
where $u=\escpr{N,U}$.
\end{proposition}
\begin{proof}
We denote $V(s,\cdot)$ by $V(s)$ for simplicity. Let us prove first that 
\begin{equation}
    \label{eq:first-1}
\dfrac{d}{ds}\Big|_{s=0} A_K( \varphi_s(S))=\int_S \dfrac{d}{ds}\Big|_{s=0} ||V(s)||_{*}\, dS.
\end{equation}
This means we can differentiate under the integral sign.

By Lemma \ref{lm:novanishing} the norm of $V(s)$ is strictly positive a.e. in $\supp(U) \cap S$. So we have 
\begin{equation}
\label{eq:fvdomination}
 \begin{aligned}
\dfrac{d}{ds} \norm{V(s)}_{*}&= \escpr{\pi_K(V(s)), \dfrac{\nabla}{ds} V(s) } \le \norm{\pi_K(V(s))}_K \norm{\dfrac{\nabla}{ds} V(s)}_{K,*}\\
&=\norm{\dfrac{\nabla}{ds} V(s)}_{*}\le \beta\, |\dfrac{\nabla}{ds} V(s)|,
\end{aligned} 
\end{equation}
for a.e. in $\supp U \cap S$, where $\beta$ is the positive constant defined in \S~\ref{sc:subper} taking $K'$ equal to the Euclidean ball centered at zero.
By Lemma \ref{lm:cdlb} there holds 
\begin{align*}
\dfrac{\nabla}{ds} V(s)=&\escpr{\nabla_{E_1(s)} U,T} T \times E_2(s)+ 2\escpr{J(U),E_1(s)}\, T \times E_2(s)+\escpr{E_1(s),T}T \times \nabla_{E_2(s)} U\\
& + 2\escpr{J(U),E_2(s)}\,E_1(s) \times T +\escpr{\nabla_{E_2(s)} U,T} E_1(s) \times T+ \escpr{E_2(s),T}\nabla_{E_1(s)} U \times T .
\end{align*}
Since $|E_1(s)|=|d\varphi_s(Z)| \le C$ and $|E_2(s)|=|d\varphi_s(E)| \le C$ for $s \in (-s_0,s_0)$ where $C>0$ is independent of $s$. Then, writing the covariant derivative $\nabla_{E_i(s)} U$ in standard coordinates, we obtain
\[
\big|\dfrac{\nabla}{ds} V(s)\big| \le \tilde{C} \,\|U\|_{C^1}
\]
a.e. in $\supp U \cap S$ for a suitable constant $\tilde{C}>0$. Here $\|U\|_{C^1}$ denotes the standard $C^1$ norm of $U$. Since $\supp U \cap S$ is compact, dominated convergence implies \eqref{eq:first-1}.

Let us compute now
\[
\dfrac{d}{ds}\Big|_{s=0} || V(s)||_*=\dfrac{d}{ds}\Big|_{s=0} \escpr{\pi_K(V(s)), V(s)}
\]
at a point $p$ of differentiability of $S$. By Remark 3.3 in \cite{2020arXiv200704683P},
\begin{align*}
\dfrac{d}{ds}\Big|_{s=0} \escpr{\pi_K(V(s)), V(s)}&= \escpr{\pi_K(V(0)), \dfrac{\nabla}{ds} \Big|_{s=0} V(s)}
= \escpr{\pi_K((\nu_h)_p), \dfrac{\nabla}{ds} \Big|_{s=0} V(s)}.
\end{align*}
Since $T$ is parallel with respect to the pseudo-hermitian connection $\nabla$ and $\escpr{Z,T}=0 $, we have 
\begin{align*}
  \dfrac{\nabla}{ds} \Big|_{s=0} V(s)=\escpr{\dfrac{\nabla}{ds} \Big|_{s=0} E_1 (s),T} \,T \times S &+ \escpr{\dfrac{\nabla}{ds} \Big|_{s=0} E_2(s),T}\, Z \times T\\
  & +\escpr{E,T} \dfrac{\nabla}{ds} \Big|_{s=0}\,  E_1(s) \times T.
\end{align*}
Lemma \ref{lm:cdlb} implies
\[
\dfrac{\nabla}{ds} \Big|_{s=0} E_1(s)=\nabla_{Z} U+ 2\escpr{J(U),Z} \,T
\]
and 
\[
\dfrac{\nabla}{ds} \Big|_{s=0} E_2(s)= \nabla_{E} U+ 2\escpr{J(U),E} \,T.
\]
Therefore, evaluating at $p$ but omitting it for clarity,
\begin{equation}
\label{eq:nUV}
\begin{split}
\frac{\nabla}{ds} \Big|_{s=0} V(s)&=\big(\escpr{\nabla_{Z} U,T}+ 2\escpr{J(U),Z}\big)\,  T \times E  \\
&\qquad+\big(\escpr{\nabla_{E} U,T}+ 2\escpr{J(U),E}\big)\,  Z \times T
+\escpr{E,T}\, \nabla_{Z}\,  U \times T\\
&= \escpr{N,T} \big(Z(\escpr{U,T}) + 2\escpr{J(U),Z}\big)\, T \times \nu_h\\
&\qquad+\big(E(\escpr{U,T})+ 2\escpr{J(U),E}\big)\, Z\times T-|N_h|\,  \nabla_{Z}  U \times T.
\end{split}
\end{equation}
We set $\pi_K(\nu_h)=\pi_Z Z+ \pi_{\nu} \nu_h$, where $\pi_Z=\escpr{\pi(\nu_h), Z}$ and $\pi_{\nu}=\escpr{\pi(\nu_h), \nu_h}$.  Notice that $T\times \nu_h=-Z$, $Z\times T=-\nu_h$ and $\escpr{\pi(\nu_h),\nabla_Z U \times T}=\escpr{J(\pi(\nu_h)), \nabla_Z U}$, then we obtain 
\begin{equation}
\label{eq:fvf1}
\begin{split}
\escpr{\pi_K(\nu_h), \dfrac{\nabla}{ds} \Big|_{s=0} V(s)}&= -\escpr{N,T} (Z(\escpr{U,T}) + 2\escpr{J(U),Z}) \pi_{Z}\\
&\qquad-(E(\escpr{U,T})+ 2\escpr{J(U),E}) \pi_{\nu} \\
&\qquad-|N_h| \escpr{J(\pi_K(\nu_h)), \nabla_Z U}\\
&=-\escpr{N,T} Z(\escpr{U,T}) \pi_Z - E(\escpr{U,T}) \pi_{\nu} \\
&\qquad-2\escpr{N,T} \escpr{J(U),\pi_K(\nu_h)}\\
&\qquad-|N_h| \escpr{J(\pi_K(\nu_h)), \nabla_Z U}.
\end{split}
\end{equation}
This implies \eqref{eq:first0}.

Let us finally prove \eqref{eq:first}. We show in Lemma~\ref{lm:C^2intbypart} that \eqref{eq:first} holds for $C^2$ surfaces. The general result follows by approximation. Following Proposition 1.20 in \cite{MR657581} or  Remark 6.1 in \cite{MR3044134}  we approximate the $(X,Y)$-Lipschitz surface $S=\{p \in \hh^1  :  f(p)=0\}$ by a family of smooth surfaces  $S_j=\{p \in \hh^1  :  f_j(p)=0\}$, where $f_j=\rho_j * f$ and  $\rho_j$ are the standard Friedrichs’ mollifiers. Since $S$ is $(X,Y)$-Lipschitz we gain that $(S_j)_0= \emptyset$. Hence $S_j$ converges to $S$ on compact subsets in $S$. Given $j \ge 1$, let  $Z^j$ be the characteristic vector field of the regular part of $S_j$ and  $\nu_h^j$ be the horizontal unit norm to $S_j$. Then we have 
\[
Z= \lim_{j\to \infty}Z^j, \quad \nu_h= \lim_{j\to \infty} \nu_h^j, \quad \lim_{j\to\infty}\nabla_{Z^j}\pi_K(\nu_h^j)=\nabla_Z\pi_K(\nu).
\]
 Since by assumption all the terms are continuous we have that the convergence is uniformly  on each compact set of $S$.
On the other hand, since $E$ and $N$ are only $L^{\infty}$ (thus in particular $L^1_{\text{loc}}$), we have that  $E^j$ and $N^j$  converge to $E$ and $N$ a.e. in $S$. Applying Lemma \ref{lm:C^2intbypart}  to the smooth surface $S_j$  and Lebesgue's dominated convergence theorem we obtain from \eqref{eq:fvf1}
\begin{align*}
\int_S \escpr{\pi(\nu_h), \dfrac{\nabla}{ds} \Big|_{s=0} V(s)} dS&=\int_S \bigg[-\escpr{N,T} Z(\escpr{U,T}) \pi_Z -E(\escpr{U,T}) \pi_{\nu}\\
& \qquad-  2\escpr{N,T} \escpr{J(U),\pi(\nu_h)} - |N_h| \escpr{J(\pi(\nu_h)), \nabla_Z U}\bigg] dS\\
&=\lim_{j\to \infty} \int_{S_j}\bigg[ -\escpr{N^j,T} Z^j (\escpr{U,T}) \pi_{Z^j} -E^j(\escpr{U,T}) \pi_{\nu^j}\\
& \qquad -  2\escpr{N^j,T} \escpr{J(U),\pi(\nu_h^j)} dS_j -  |N_h^j| \escpr{J(\pi(\nu_h^j)), \nabla_{Z^j} U} \bigg] dS_j \\
&=\lim_{j\to \infty} \int_{S^j}  u \escpr{\nabla_{Z^j} \pi(\nu_h^j),Z^j} dS_j\\
&= \int_{S} u \escpr{\nabla_Z \pi(\nu_h),Z} dS,
\end{align*}
since $\pi_K$ is $C^1$. This concludes the proof.
\end{proof}

\begin{lemma}
\label{lm:C^2intbypart}
Let $S \subset \hh^1$ be an oriented $C^2$ surface with $S_0=\emptyset$. Let $U$ be a $C^1$ vector field with compact support and normal component $u=\escpr{U, N}\in C_0^1(S)$. Then we have 
\begin{equation}
\label{eq:fvflemma}
\begin{split}
\int_{S} u \escpr{\nabla_Z \pi(\nu_h),Z} dS&=\int_S \bigg[ -\escpr{N,T} Z(\escpr{U,T}) \pi_Z -E(\escpr{U,T}) \pi_{\nu} \\
&-  2\escpr{N,T} \escpr{J(U),\pi(\nu_h)}
-  |N_h| \escpr{J(\pi(\nu_h)), \nabla_Z U} \bigg]dS,
\end{split}
\end{equation} 
 where $\pi_Z=\escpr{\pi_K,Z}$ and $\pi_{\nu}=\escpr{\pi_K,\nu_h}$.
\end{lemma}
\begin{proof}
Since both sides of \eqref{eq:fvflemma} depend linearly on $U$, we consider the horizontal $U_h=\escpr{U,Z}\, Z+ \escpr{U,\nu_h}\, \nu_h$ and vertical $U_v=\escpr{U,T}\, T$ components of $U$.

We start considering the horizontal component $U_h$ and we denote by $I_h$ the corresponding right hand side of \eqref{eq:fvflemma}. Thus we obtain
\begin{align*}
I_h&=\int_S \bigg\{-2\escpr{N,T} \escpr{J(U_h),\pi(\nu_h)}
-|N_h| \escpr{J(\pi(\nu_h)), \nabla_Z U_h}\bigg\} dS\\
&=\int_S -2\escpr{N,T} \escpr{J(U_h),\pi(\nu_h)}\,dS
\\
& \hspace{4em}
+\int_S  \bigg\{|N_h| Z\escpr{\pi(\nu_h), J(U_h)}
+ |N_h|\escpr{ \nabla_Z \pi(\nu_h), J(U_h)}\bigg\}\,dS,
\end{align*}
since $\nabla_W J(V)=J(\nabla_W V)$ (as $\nabla_W J=0$) for each pair of vector fields $W,V$.
By Proposition \ref{pr:divt} we have
\begin{align*}
\int_S |N_h| Z(\escpr{\pi(\nu_h), J(U_h)}) dS&=-\int_S \escpr{\pi(\nu_h), J(U_h)} Z(|N_h|) dS\\
&-\int_S |N_h|\escpr{\pi(\nu_h), J(U_h)} \divv_S(Z) dS,
\end{align*}
where  $\divv_S(Z)=\escpr{D_E Z,E}=-\escpr{N,T} \theta(E)-2\escpr{N,T}|N_h|$.
Moreover, we have 
\[
\escpr{ \nabla_Z \pi(\nu_h), J(U_h)}=-\escpr{U,\nu_h}\escpr{ \nabla_Z \pi(\nu_h), Z},
\]
since $ \escpr{ \nabla_Z \pi(\nu_h), \nu_h}=0$, by Remark 3.3 in \cite{2020arXiv200704683P}. Hence we have
\begin{equation}
\label{eq:fvfhf}
\begin{split}
I_h&= \escpr{J(U_h),\pi(\nu_h)}(-2\escpr{N,T}^3 -Z(|N_h|)  + |N_h|\escpr{N,T} \theta(E) )\\
& \hspace{4em}+ |N_h| \escpr{U,\nu_h}\escpr{ \nabla_Z \pi(\nu_h), Z}.
\end{split}
\end{equation}

Now we consider the vertical component of the variational vector field $U_v=\escpr{U,T}T$. We denote by $I_v$ the left hand side of  \eqref{eq:fvflemma}. Thus
\begin{equation}
\label{eq:fvft}
\begin{aligned}
I_v=\int_S -\escpr{N,T} \pi_Z  Z(\escpr{U,T})  - E(\escpr{U,T}) \pi_{\nu} \, dS\\
\end{aligned}
\end{equation}
By Proposition \ref{pr:divt} we get 
\begin{align*}
-\int_S \escpr{N,T} \pi_Z  Z(\escpr{U,T}) dS &=\int_S (Z(\escpr{N,T}) \pi_Z+ \escpr{N,T} Z(\pi_Z))  \escpr{U,T} dS\\
 &- \int_S \escpr{U,T} \escpr{N,T} \pi_Z (\escpr{N,T} \theta(E))+2\escpr{N,T} |N_h|) dS\\
 &=\int_S \escpr{U,T} Z(\escpr{N,T}) \pi_Z dS\\
 &+ \int_S \escpr{N,T} \escpr{U,T} (\escpr{\nabla_Z \pi(\nu_h),Z}- \pi_{\nu} \theta(Z))  dS\\
 &- \int_S \escpr{U,T} \escpr{N,T} \pi_Z (\escpr{N,T} \theta(E))+2\escpr{N,T} |N_h|) dS,
\end{align*}
and 
\begin{align*}
-\int_S  \pi_{\nu}  E(\escpr{U,T}) dS &=\int_S E(\pi_{\nu})  \escpr{U,T} dS+ \int_S \escpr{U,T}  \escpr{N,T }\pi_{\nu}  \theta(Z) dS\\
 &=\int_S \pi_Z \theta(E)  \escpr{U,T} dS+ \int_S \escpr{U,T}  \escpr{N,T }\pi_{\nu}  \theta(Z) dS
 \end{align*}
Adding the previous terms we obtain 
\begin{equation}
\label{eq:fvft2}
\begin{split}
I_v
&=\int_S \escpr{U,T} Z(\escpr{N,T}) \pi_Z dS+ \int_S \escpr{N,T} \escpr{U,T} (\escpr{\nabla_Z \pi(\nu_h),Z} )  dS \\
 &- \int_S 2 \escpr{U,T} \pi_Z \escpr{N,T}^2  |N_h|) dS + \int_S \pi_Z |N_h|^2 \theta(E)  \escpr{U,T} dS\\
\end{split}
\end{equation}
Since tangential variations do not change the first variation formula,  we consider a normal variational vector field $U=u N$ with $u\in C_0^1(S)$ so that $\escpr{U,Z}=0$, $\escpr{U,\nu_h}=u |N_h|$ and $\escpr{U,T}=u \escpr{N,T}$. Then adding the integral of the  horizontal term \eqref{eq:fvfhf} and the vertical term \eqref{eq:fvft2} we obtain
\begin{equation}
\label{eq:fvff}
\begin{aligned}
I_h+I_v&=\int_S -u \, \pi_Z |N_h| (-2\escpr{N,T}^3 -Z(|N_h|)  + |N_h|\escpr{N,T} \theta(E) ) dS\\
&  +\int_S u \, \pi_Z ( \escpr{N,T}Z(\escpr{N,T})-2 |N_h| \escpr{N,T}^3   + |N_h|^2\escpr{N,T} \theta(E) ) dS\\
& + \int_S u  (\escpr{N,T}^2+|N_h|^2) \escpr{\nabla_Z \pi(\nu_h),Z}.
\end{aligned}
\end{equation}
Since $|N_h|Z(|N_h|)+\escpr{N,T} Z(\escpr{N,T})=0$ we obtain 
\[
I_h+I_v=\int_S u \escpr{\nabla_Z \pi(\nu_h),Z}
\]
thus proving \eqref{eq:fvflemma}.
\end{proof}

To obtain the first variation formula \eqref{eq:first} we had to assume that the derivatives in the $Z$-direction of the vector fields $\nu_h$ and $Z$ exist and are continuous functions on $S$. Let us see that the area-stationary property of $S$ implies this regularity property. We follow the arguments in \cite{MR4314055} and \cite{MR3984100}. This result was first proven in the sub-Riemannian case by Nicolussi and Serra-Cassano \cite{MR3984100}.  

\begin{definition}
Let $S \subset \hh^1$ be a $(X,Y)$-Lipschitz surface with boundary $\partial S$. We  say $S$ is area-stationary if, for any $C^1$ vector field $U$ with compact support such that $ \supp(U) \cap \partial S= \emptyset$, and associated one-parameter group of diffeomorphisms $\{\varphi_s\}_{s \in \rr}$, we have 
\[
\dfrac{d}{ds}\Big|_{s=0} A_K( \varphi_s(S))=0.
\]
\end{definition}

\begin{remark}
\label{rk:horcur}
 Let $D$ be a domain in the vertical plane $\{y=0\}$ and let $u: D\subset \rr^2 \to \rr$ be a Lipschitz function. Since the vector field $Y$ is a unit normal to the plane, the intrinsic
graph $\text{Gr}(u)$ is parametrized by 
\[
(x,t)\to (x,u(x,t),t-xu(x,t)).
\]
Let $\ga(s)=(x,t)(s)$ be a Lipschitz curve in $D$. Its lifting
\[
\Gamma(s)=(x,u(x,t), t-x u(x,t))(s) \subset \text{Gr}(u)
\]
is also Lipschitz and 
\[
\Gamma'(s)=x' X +(x'u_x+t'u_t)Y+(t'-2ux')T
\]
a.e. in $s$. In particular horizontal curves in $\text{Gr}(u)$  satisfy the ordinary differential equation
\begin{equation}
\label{eq:horode}
t'=2 u(x,t)\, x'.
\end{equation}
\end{remark}

\begin{theorem}
\label{th:straightfoliation}
Let $K\in C^2_+$ be a convex body with $0 \in \intt(K)$. Let $S \subset \hh^1$ be an area-stationary $(X,Y)$-Lipschitz surface. Then $S$ is an $\hh$-regular surface  foliated by horizontal straight lines.
\end{theorem}

\begin{proof}
Let $p$ in $S$. Since $S$ is $(X,Y)$-Lipschitz, by Theorem \ref{th:iftLip}, there exist an open ball $B_r(p)$ and a Lipschitz function $u:D \to \rr$ such that $S \cap B_r(p)=\text{Gr}(u)$ where $\text{Gr}(u)=\{(x,u(x,y),t-xu(x,t))\in \hh^1 : (x,t) \in D\}$. Setting $\pi_K=(\pi_1,\pi_2)$ the area functional is given by
\begin{align*}
A(\text{Gr}(u))=\int_D  (u_x + 2 u u_t)   \pi_1(u_x + 2 u u_t, -1) - \pi_2(u_x + 2 u u_t, -1)  \, dx dt.
\end{align*} 
Given $v \in C_0^{\infty}(D)$, a straightforward computation shows that 
\begin{equation}
\label{eq:fv0}
\dfrac{d}{ds}\Big|_{s=0} A(\text{Gr}(u+sv))= \int_D (v_x+ 2v u_t+ 2 u v_t) M dx dt,
\end{equation}
where 
\begin{equation}
\label{eq:M}
M=F(u_x+2uu_t),
\end{equation}
 and $F$ is the function
\begin{equation}
\label{def:F}
F(x)= \pi_1(x,-1) + x \dfrac{\ptl \pi_1}{\partial x} (x,-1)- \dfrac{\partial \pi_2}{\partial x} (x,-1).
\end{equation}

Let $\Gamma(s)$ be a characteristic curve passing through $p$  in $\text{Gr}(u)$. Let $\ga(s)$ be the projection of $\Gamma(s)$ onto the $xt$-plane. By composition with a left-translation we may assume that $(0,0) \in D$ is the projection of $p$ to the $xt$-plane. We  parameterize $\ga$ by  $s\to (s,t(s))$. By Remark~\ref{rk:horcur} the curve $s \to (s,t(s))$  satisfies the ordinary differential equation $t'=2u$. For $\eps$ small enough, Picard-Lindel\"of's theorem implies the existence of $r>0$ and  a solution $t_{\eps}: ]-r, r[\to \rr$ of the Cauchy problem 
\begin{equation}
\label{eq:teps0}
\begin{cases}
t_{\eps}'(s)=2u(s, t_{\eps}(s)),\\
t_{\eps}(0)=\eps.
\end{cases}
\end{equation}
We define $\ga_{\eps}(s)=(s,t_{\eps}(s))$  so that $\ga_0=\ga$. By Lemma \ref{lm:teps>0} we obtain that $G(\xi,\eps)=(\xi,t_{\eps}(\xi))$ is a biLipschitz homeomorphisms where the determinant of the Jacobian of $G$ is given by $\partial t_{\eps} (s)/ \partial \eps > C>0$ for each $s \in ]-r,r[$ and  a.e.~in $\eps$ .

Any function $\varphi$ defined on $D$ can  be considered as a function of the variables $(\xi,\eps)$ by making $\tilde{\varphi}(\xi,\eps)=\varphi(\xi,t_{\eps}(\xi))$. Since the function $G$ is $C^1$ with respect to $\xi$ we have 
\[
 \frac{\partial \tilde{\varphi}}{\partial \xi}= \varphi_x+  t_{\eps}' \, \varphi_t =\varphi_x+  2 u  \varphi_t.   
 \]
Furthermore, by \cite[Theorem 2 in Section 3.3.3]{Evans2015} or \cite[Theorem 3]{MR1201446}, we may apply the change of variables formula for Lipschitz maps. Assuming that the support of $v$  is contained in a sufficiently small neighborhood of $(0,0)$, we can express the integral \eqref{eq:fv0} as
\begin{equation}
\label{eq:fv1}
 \int_I \bigg(\int_{-r}^{r} \big( \frac{\partial \tilde{v}}{\partial \xi}+ 2 \tilde{v} \, \tilde{u}_t\big) \tilde{M}\,  d\xi \bigg)\, d\eps=0,
\end{equation}
where $I$ is a small interval containing $0$.  In equation \eqref{eq:fv1} we used the area stationary assumption. Instead of $\tilde{v}$ in \eqref{eq:fv1} we consider the function $\tilde{v} h/(t_{\eps+h}-t_{\eps})$, where $h$ is a small enough parameter. Then we obtain 
\begin{align*}
\dfrac{\partial }{\partial \xi} \left({  \dfrac{\tilde{v} h }{t_{\eps+h}-t_{\eps}}} \right)&= \frac{\partial \tilde{v}}{\partial \xi}\, \dfrac{ h }{t_{\eps+h}-t_{\eps}}-  \tilde{v} h \dfrac{t_{\eps+h}'-t_{\eps}'} {(t_{\eps+h}-t_{\eps})^2}\\
&=\frac{\partial \tilde{v}}{\partial \xi}\, \dfrac{ h }{t_{\eps+h}-t_{\eps}}-  2 \tilde{v} h \dfrac{u(\xi, t_{\eps+h}(\xi))-u(\xi,t_{\eps}(\xi)} {(t_{\eps+h}-t_{\eps})^2},
\end{align*}
that tends to
\[
\left(\frac{\partial t_{\eps}}{\partial \eps}\right)^{-1} \left( \frac{\partial \tilde{v}}{\partial \xi} -2  \tilde{v} \tilde{u}_t \right) \qquad a.e. \ \text{in} \ \eps,
\]
when $h$ goes to $0$. Putting $\tilde{v} h/(t_{\eps+h}-t_{\eps})$ in \eqref{eq:fv1} instead of $\tilde{v}$ we get
\[
\int_I \bigg(\int_{-r}^{r}  \dfrac{h \frac{\partial t_{\eps}}{\partial \eps}}{t_{\eps+h}-t_{\eps}}\bigg(\frac{\partial \tilde{v}}{\partial \xi}+ 2 \tilde{v} \, \big(\tilde{u}_t- \dfrac{\tilde{u}(\xi, {\eps+h})-\tilde{u}(\xi,{\eps})} {(t_{\eps+h}-t_{\eps})}\big) \bigg) \tilde{M}  \, d\xi \bigg) d\eps=0.
\]
Using Lebesgue's dominated convergence theorem and letting $h\to 0$ we have
\begin{equation}
\label{eq:tildefvf}
\int_I \left(\int_{-r}^{r}  \frac{\partial \tilde{v}}{\partial \xi} \tilde{M}  \, d\xi \right) d\eps=0.
\end{equation}
Let $\eta: \rr \to \rr$ be a positive function compactly supported in $I$ and for $\rho>0$ we consider the family $\eta_{\rho}(x)=\rho^{-1} \eta(\tfrac{x-\eps_0}{\rho})$, that weakly converge to the Dirac delta distribution. Putting  the test functions $\eta_{\rho}(\eps) \psi(\xi)$ in  \eqref{eq:tildefvf} and letting $\rho\to 0$ we get
\begin{equation}
\label{eq:fvacc}
\int_{-r}^{r}  \psi'(\xi) \tilde{M}(\xi,\eps_0)  \, d\xi=0,
\end{equation}
for each $\psi \in C^{\infty}_0((-r,r))$ for a.e. $\eps_0$ in $I$. Since $F$ is $C^1$ and the distributional derivatives of a Lipschitz function belongs $L^{\infty}$ we gain that $M$ defined in \eqref{eq:M}  is  $L^{\infty}(D)$. In particular we have that $M$ belongs $L^1_{\text{loc}}(D)$, thus by Fubini's Theorem also  $\tilde{M}(\cdot,\eps_0)$ belongs to $L^1_{\text{loc}}((-r,r))$ for a.e. $\eps_0$ in $I$. By equation \eqref{eq:fvacc} we gain that $\tilde{M}(\cdot,\eps_0)$ belongs to $W^{1,1}((-r,r))$ with   $\partial_{\xi} \tilde{M}=0$ a.e. in $(-r,r)$. Then by \cite [Theorem  8.2]{MR2759829} we gain that $\tilde{M}(\cdot,\eps_0)$ is absolutely continuous and $\partial_{\xi} \tilde{M}=0$ a.e. in $(-r,r)$ thus $\tilde{M}(\cdot,\eps_0)$ is constant in $\xi$ for a.e. $\eps_0 \in I$. Therefore $M$ is constant along $\ga_{\eps_0}(s)=(s,t_{\eps_0}(s))$ for a.e. $\eps_0$ in $I$. By Lemma 3.2 in \cite{MR4314055} $F$ is a $C^1$ invertible function, therefore  also $g(s)=(u_x+2u u_t)_{\ga_{\eps_0}(s)}=F^{-1}(M)$ is constant in $s$ for a.e. $\eps_0$ in $I$.
This shows that horizontal normal given by
\begin{equation}
\label{eq:nuhlc}
\nu_h=\dfrac{(u_x+2u u_t)X-Y}{\sqrt{1+(u_x+2u u_t)^2}}
\end{equation}
is constant along the characteristic curves, thus also $Z=-J(\nu_h)$ is constant. Hence the characteristic curves of $S$ are straight lines. Here we follow the approach developed by \cite{MR3984100}.
Moreover, since $2g(s)=2(u_x+2u u_t)_{\ga_{\eps_0}(s)}=t_{\eps}''(s)$ is constant in $s$ we have that
$t_{\eps}(s)$ is a polynomial of the second order given by 
\[
t_{\eps}(s)=\eps+a(\eps) s+ b(\eps)  s^2
\]
where $a(\eps)=2u(0,\eps)$ that is Lipschitz continuous and $b(\eps)=(u_x+2u u_t)(0,\eps)=(u_x+2u u_t)(s,\eps)$. Furthermore, choosing $s>0$ we can easily prove that  $b(\eps)$ is also a Lipschitz function function in $\eps$. Hence in particular the horizontal normal $\nu_h$  given by \eqref{eq:nuhlc} is continuous, then the surface is an $\hh$-regular surface.  
\end{proof} 

\begin{remark}
\label{rk:regularity}
Notice that a $(X,Y)$-Lipschitz area-stationary surface is $\hh$-regular and  its horizontal normal $\nu_h$ is $C^{\infty}$ in the $Z$-direction.
\end{remark}

\begin{lemma}
\label{lm:teps>0}
With the previous notation, there exists a bi-Lipschitz homeomorphism $G(\xi,\eps)=(\xi,t_{\eps}(\xi))$. Moreover, there exists a constant $C>0$ such that $\partial t_{\eps} (s)/ \partial \eps > C$ for each $s \in ]-r,r[$ and  a.e.~in $\eps$.
\end{lemma}

\begin{proof}
Here we exploit an argument similar to the one developed in \cite{MR3984100}.
By Theorem~2.8 in \cite{MR2961944} we gain that $t_{\eps}$ is Lipschitz with respect to $\eps$ with Lipschitz constant less than  or equal to $e^{Lr}$.  Fix $s \in ]-r,+r[$, the inverse of the function $\eps \to t_{\eps}(s)$ is given by $\bar{\chi}_{t}(-s)= \chi_{t}(-s)$ where $\chi_t$ is the unique solution of the following Cauchy problem
\begin{equation}
\label{eq:tau}
\begin{cases}
\chi_{t}'(\tau)=2u(\tau, \chi_{t}(\tau))\\
\chi_{t}(s)=t.
\end{cases}
\end{equation}
Again by Theorem~2.8 in \cite{MR2961944} we have that $\bar{\chi}_{t}$ is Lipschitz continuous with respect to $t$, thus the function $\eps \to t_{\eps}$ is a locally biLipschitz homeomorphisms.

We  consider the following Lipschitz coordinates 
\begin{equation}
\label{eq:Gcc}
G(\xi,\eps)=(\xi,t_{\eps}(\xi))=(s,t)
\end{equation}
around the characteristic curve passing through  $(0,0)$. Notice that, by the uniqueness result for  \eqref{eq:teps},  $G$ is injective. 
Given $(s,t)$ in the image of $G$ and using the inverse function $\bar{\chi}_{t}$ defined in \eqref{eq:tau} we find $\eps$ such that $t_{\eps}(s)=t$, therefore $G$ is surjective. By the Invariance of Domain Theorem \cite{MR1511684},  $G$  is a homeomorphism. By the uniqueness result of the Cauchy problem \eqref{eq:teps} we get that the map $\eps \to t_{\eps}(s)$ is not decreasing in $\eps$, then we have 
\begin{equation}
\label{eq:teps>0}
\frac{\partial t_{\eps}(s)}{\partial \eps} \ge0
\end{equation}
for a.e. in $\eps$.
The differential  of $G$ is defined by  
\begin{equation}
\label{eq:JacG}
DG=  \left(\begin{array}{cc} 1 & 0\\ t_{\eps}' & \frac{\partial t_{\eps}}{\partial \eps} \end{array}\right)
\end{equation}
almost everywhere in $\eps$ and the differential of $G^{-1}$ is given by 
\begin{equation}
\label{eq:JacGinv}
D G^{-1}= (\tfrac{\partial t_{\eps}}{\partial \eps})^{-1} \left(\begin{array}{cc} \frac{\partial t_{\eps}}{\partial \eps} & 0\\ -t_{\eps}' & 1 \end{array}\right)
\end{equation}
almost everywhere in $\eps$.  Since $G^{-1}$ is Lipschitz we gain that there exists a constant $C>0$ such that
\[
|\mathbf{J}_{G^{-1}}|=|\det(D G^{-1})|= |\tfrac{\partial t_{\eps}}{\partial \eps}|^{-1} \le \dfrac{1}{C}
\] 
a.e. in  $\eps$. Thus, by \eqref{eq:teps>0} we deduce that $\partial t_{\eps} / \partial \eps>C>0$ a.e. in $\eps$. 
\end{proof}

\section{A Codazzi-like equation for $(X,Y)$-Lipschitz minimal surfaces}
\label{sec:codazzi}

In this section we shall show that, given an area-stationary surface $S$, the function $\escpr{N,T}/|N_h|$ satisfies a differential equation along almost every characteristic curve on $S$.

We first prove a technical result similar to Lemma~4.2 in \cite{MR3406514}. We include the proof, with only slight differences, for completeness.

\begin{lemma}
\label{lm:ODECodazzi}
Given $a,b \in \rr$, the only solution of equation
\begin{equation}
\label{eq:ODECodazzi}
y''-6y'y+4y^3=0
\end{equation}
about the origin with initial conditions $y(0) = a$, $y'(0) = b$, is
\begin{equation}
\label{eq:uab}
y_{a,b} (s)=\dfrac{a-(2a^2-b)s}{1-2as+(2a^2-b)s^2}. 
\end{equation}
Moreover, we have 
\begin{equation}
\label{eq:uabq}
y^2_{a,b} (s)-y_{a,b} '  (s)= \dfrac{a^2-b}{(1-2as+(2a^2-b)s^2)^2}
\end{equation}
If $y_{a,b}$ is defined for every $ s\in \rr$ then either $a^2-b>0$ or $y_{a,b}\equiv 0$.
\end{lemma}
\begin{proof}
By the uniqueness of solutions for ordinary differential equations we know that there exists a unique solution \eqref{eq:ODECodazzi}. 
Since we have 
\[
y_{a,b}'=\dfrac{b-2as(2a^2-b)+(2a^2-b)^2s^2}{(1-2as+(2a^2-b)s^2)^2}
\]
and
\[
y_{a,b}''=\dfrac{2(a-(2a^2-b)s)(3b-2as(2a^2-b)+(2a^2-b)^2 s^2-2a^2 )}{(1-2as+(2a^2-b)s^2)^3},
\]
a straightforward computation shows that $y_{a,b} (s)$ solves  \eqref{eq:ODECodazzi} and satisfies \eqref{eq:uabq}.

Let us write $y_{a,b}={p(s)}/{r(s)}$ where $p(s)=a-(2a^2-b)s$ and $r(s)=1-2as+(2a^2-b)s^2$  If $y_{a,b}$ is defined for every $ s\in \rr$, then  there are two possibilities: $r(s)$ has  no real zeroes or $r(s)$ has at least a zero at $s_0 \in \rr$. In the first case the discriminant is $4(b-a^2)$ is negative. In the second case $r(s_0)=0$ we must also have $p(s_0) = 0$ in order to have $y_{a,b}(s)$ well defined at $s_0$. Hence $y_{a,b}(s)$ can be expressed as the quotient of a constant over a degree one polynomial. Then by \eqref{eq:uab} we get that $y_{a,b}=a (1-2as)^{-1}$ which has a pole unless $a=0$, hence $y_{a,b}(s)\equiv0$.
\end{proof}

\begin{remark}
\label{rk:dilationODE}
If $f(s)$ is a solution of \eqref{eq:ODECodazzi}, then for each positive constant $\lambda$ the function
$f_{\lambda}(s)=\lambda^{-1}f(\tfrac{s}{\lambda})$
is still a solution of \eqref{eq:ODECodazzi}.
\end{remark}

\begin{remark}
Let $f$ be a solution of \eqref{eq:ODECodazzi}, then  $f$ belongs to $C^{\infty}$ class.
Indeed setting $y_1=f$  and $y_2=y_1'$ we have that \eqref{eq:ODECodazzi} is equivalent to 
\[
\begin{pmatrix} y'_1\\ y'_2 \end{pmatrix} =F(y_1,y_2)= \begin{pmatrix} y_2 \\ -6 y_1 y_2-4y_1^3 \end{pmatrix}.
\]
Since $F$ is $C^{\infty}$ we obtain that $y_1=f$ is smooth.
\end{remark}

\begin{proposition}
\label{pr:satode}
Let $S$ be a complete oriented area-stationary $(X,Y)$-Lipschitz surface. Then along any  arc-length parameterized geodesic $\bar{\ga}_{\eps}(s)$  in $S$, the function $\escpr{N,T}/|N_h|(\bar{\ga}_{\eps}(s))$ satisfies the ordinary differential equation \eqref{eq:ODECodazzi} for a.e. $\eps$.
Furthermore, $\escpr{N,T}/|N_h|(\bar{\ga}_{\eps}(s))$ is smooth in $s$ for a.e. $\eps$. 
\end{proposition}
\begin{proof}
Let $p$ in $S$. Since $S$ is $(X,Y)$-Lipschitz, Theorem \ref{th:iftLip} implies the existence of an open ball $B_r(p)$ and of a Lipschitz function $u:D \to \rr$ such that $S \cap B_r(p)=\text{Gr}(u)$. 
Let $\Gamma(s)$ be a characteristic curve passing through $p$  in $\text{Gr}(u)$. Let $\ga(s)$ be the projection of $\Gamma(s)$ onto the $xt$-plane, and $(0,0) \in D$ the projection of $p$ to the $xt$-plane. We  parameterize $\ga$ by  $s\to (s,t(s))$. By Remark \ref{rk:horcur} (see also \cite[Remark 2.5]{MR4314055}) the curve $s \to (s,t(s))$  satisfies the ordinary differential equation $t'=2u$ and 
\[
\Ga'(s)=X+(u_x+2u u_t)Y.
\]
As computed in \S~2.6 in \cite{MR4314055}, at smooth points of the graph of $u$, the unit normal can be computed as $N=\tilde{N}/|\tilde{N}|$, where
\[
\tilde{N}=(u_x+2uu_t)-Y+u_tT.
\]
The quantity $|\tilde{N}|$ is the Riemannian Jacobian of the parameterization of the intrinsic graph of $u$ by coordinates $(x,t)$. So we have
\begin{equation}
\label{eq:NhdS}
|N_h|\, dS=\sqrt{1+(u_x+2uu_t)^2}\,dx\,dt.
\end{equation}
Since  we have
$$\nu_h=\dfrac{(u_x+2u u_t)X-Y}{\sqrt{1+(u_x+2u u_t)^2}} \quad \text{and} \quad Z=-J(\nu_h)$$
we get $Z=- \Ga'(s)/|\Ga(s)|$. For $\eps$ small enough, Picard-Lindel\"of's theorem implies the existence of $r>0$ and  a solution $t_{\eps}:\ ]-r,r[\to \rr$ of the Cauchy problem 
\begin{equation}
\label{eq:teps}
\begin{cases}
t_{\eps}'(s)=2u(s, t_{\eps}(s)),\\
t_{\eps}(0)=\eps.
\end{cases}
\end{equation}
We define $\ga_{\eps}(s)=(s,t_{\eps}(s))$  so that $\ga_0=\ga$. 
Since $S$ is area-stationary we have that $(u_x+2u u_t)$ is constant along $\ga_{\eps}(s)$. Moreover
\[
t_{\eps}''(s)=2(u_x+2u u_t)(\ga_{\eps}(s))=2 b(\eps)=2(u_x+2u u_t)(0,\eps)
\]
is constant as a a function of $s$. Thus we have
\begin{equation}
\label{eq:taylorex}
t_{\eps}(s)=\eps+a(\eps) s+ b(\eps)  s^2,
\end{equation}
where $a(\eps)=2u(0,\eps)$. Choosing $s>0$ in \eqref{eq:taylorex} we can easily prove that  $b(\eps)$, that a priori is only continuous, is also a Lipschitz function. By equation (7) in   \cite[Theorem 3.7]{MR3984100} we have 
\begin{equation}
\label{eq:crossderivatives}
\dfrac{\partial}{\partial \eps} \dfrac{\partial}{\partial s} t_{\eps}(s)= \dfrac{\partial}{\partial s} \dfrac{\partial}{\partial \eps} t_{\eps}(s)
\end{equation}
a.e. in $\eps$, where the equality has to be interpreted in the sense of distributions.
Putting \eqref{eq:teps} in the left hand side of \eqref{eq:crossderivatives} and applying the chain rule for Lipschitz functions (see  \cite[Remark 3.6]{MR3984100})   we get 
\[
2 u_t(s,t_{\eps}(s)) (1+a'(\eps)s+b'(\eps) s^2)= (a'(\eps)+ 2 b'(\eps)s)
\]
a.e. in $\eps$.
Therefore we get
\[
u_t(s,t_{\eps}(s)) =\dfrac{ \tfrac{a'(\eps)}{2}+  b'(\eps)s}{ (1+a'(\eps)s+b'(\eps) s^2)},
\]
a.e. in $\eps$, since by Lemma \ref{lm:teps>0} we have $\partial t_{\eps} / \partial \eps >0$ a.e. in $\eps$.  Since we have  $Z=- \Ga'(s)/|\Ga(s)|$ we consider $\tilde{\ga}_{\eps}(s)= \ga_{\eps}(-s)$. Then we have that 
\[
u_t(\tilde{\ga}_{\eps}(s))=\dfrac{ \tfrac{a'(\eps)}{2}-  b'(\eps)s}{ (1-a'(\eps)s+b'(\eps) s^2)}
\]
solves the equation \eqref{eq:ODECodazzi} with initial condition $y(0)=a'(\eps)/2$ and 
$y'(0)=\tfrac{a'(\eps)^2}{2}- b'(\eps)$ for a.e. $\eps$. Moreover we have 
\[
t_{\eps}(-s)=\eps-a(\eps) s+ b(\eps)  s^2
\]
For each $\eps$ fixed we have $b(\eps)=(u_x+2u u_t)(\tilde{\ga}_{\eps})$ is constant, let 
\[
\bar{\ga}_{\eps} (s)=\tilde{\ga}_{\eps}\left( s/\sqrt{1+b(\eps)^2} \right)
\]
be an arc-length parametrization of $\tilde{\gamma}_{\eps}$. Then  Remark \ref{rk:dilationODE} shows that also 
\[
\escpr{N,T}/|N_h|(\bar{\ga}_{\eps})=\dfrac{u_t}{\sqrt{1+(u_x+2u u_t)^2}} (\bar{\ga}_{\eps})
\]
is a solution of \eqref{eq:ODECodazzi} a.e. in $\eps$.
\end{proof}

\section{Second Variation Formula}
\label{sec:second}

In this section, we compute the second variation formula. First of all we need the following Lemma
\begin{lemma}
\label{lm:cd2}
Let $U$ be a $C^2$ horizontal vector field in $\hh^1$ with associated flow $\{\varphi_s\}_{s\in\rr}$. Let $p \in \hh^1$ and $e \in T_p \hh^1$. Define  the smooth curve $\beta(s)=\varphi_s(p)$ and the smooth vector field $E(s)=(d\varphi_s)_p(e)$ along $\beta$. Then we have 
\[
\dfrac{\nabla^2}{ds^2}\Big|_{s=0} E(s)= \nabla_e \nabla_U U+2\, \escpr{J(\nabla_{U_p} U),e}\,T_p + 2\, \escpr{J(U_p),\nabla_{e} U}\,T_p.
\]
\end{lemma}

\begin{proof}
From \eqref{eq:1stnabla} we get
\[
\dfrac{\nabla^2}{ds^2}\Big|_{s=0} E(s)=\dfrac{\nabla}{ds}\Big|_{s=0} \big(  \nabla_{E(s)} U_s+2\escpr{J(U_s),E(s)}\,T\big).
\]
On the one hand we have
\begin{align*}
\dfrac{\nabla}{ds}\Big|_{s=0}   \nabla_{E(s)} U_s&= \dfrac{\nabla}{ds}\Big|_{s=0} \sum_{i=1}^3 g_i(s) \nabla_{(\frac{\partial }{\partial x_i})_{\beta(s)} }U_s \\
&= \sum_{i=1}^3 g_i'(0) \nabla_{(\frac{\partial }{\partial x_i})_{p}} U+ g_i(0)  \dfrac{\nabla}{ds}\Big|_{s=0} \nabla_{(\frac{\partial }{\partial x_i})_{p}} U  \\
&= \sum_{i=1}^3 e(f_i) \nabla_{(\frac{\partial }{\partial x_i})_p } U+ e_i  \nabla_U \nabla_{(\frac{\partial }{\partial x_i})_{p}} U .
\end{align*}
Notice that  
\begin{align*}
\nabla_e(\nabla_U U )=\nabla_e \nabla_{\sum_{i=1}^3 f_i \frac{\partial }{\partial x_i}} U&= \sum_{i=1}^3 e(f_i) \nabla_{(\frac{\partial }{\partial x_i})_p } U+ f_i(p)\, \nabla_e \nabla_{(\frac{\partial }{\partial x_i}) } U\\
&=\sum_{i=1}^3 e(f_i) \nabla_{(\frac{\partial }{\partial x_i}) _p} U+ \sum_{i,j=1}^3 f_i(p)\, e_j\nabla_{(\frac{\partial }{\partial x_j})} \nabla_{(\frac{\partial }{\partial x_i}) } U\\
&=\sum_{i=1}^3 e(f_i) \nabla_{(\frac{\partial }{\partial x_i})_p } U+ \sum_{i,j=1}^3 f_i(p)\, e_j\nabla_{(\frac{\partial }{\partial x_i})} \nabla_{(\frac{\partial }{\partial x_j}) } U\\
&=\sum_{i=1}^3 e(f_i) \nabla_{(\frac{\partial }{\partial x_i})_p } U+ \sum_{i=1}^3  e_i\nabla_{U} \nabla_{(\frac{\partial }{\partial x_i}) } U,
\end{align*}
where we use that the Riemann tensor of $\nabla$ vanishes
\[
0=R\big(\frac{\partial }{\partial x_j} , \frac{\partial }{\partial x_i}\big)\, U= \nabla_{\left(\frac{\partial }{\partial x_i}\right)} \nabla_{\left(\frac{\partial }{\partial x_j}\right) } U- \nabla_{\left(\frac{\partial }{\partial x_j}\right)} \nabla_{\left(\frac{\partial }{\partial x_i}\right) } U
\]
Therefore
\begin{equation}
\label{eq:cd1t}
\dfrac{\nabla}{ds}\Big|_{s=0}   \nabla_{E(s)} U_s= \nabla_e(\nabla_U U ) + \nabla_e U'
\end{equation}
On the other hand, since $\nabla J= 0$ we have 
\begin{equation}
\label{eq:cdt2}
\begin{aligned}
\dfrac{\nabla}{ds}\Big|_{s=0}  2\escpr{J(U_s),E(s)}\,T&= 2\escpr{J(\nabla_U U),e}\, T+ 2\escpr{J(U),\dfrac{\nabla}{ds}\Big|_{s=0} E(s) }\,T\\
&= 2\escpr{J(\nabla_U U),e}\, T + 2\escpr{J(U),\nabla_e U }\,T,
\end{aligned}
\end{equation}
where we have used once again Lemma \ref{lm:cdlb}  and the fact that $J(U)$ is horizontal.
Finally adding \eqref{eq:cd1t} and \eqref{eq:cdt2} we get the result.
\end{proof}

Now we compute the second variation formula

\begin{theorem}
\label{th:secondvf}
Let $K\in C^2_+$ be a convex body with $0 \in \intt(K)$. Let $S \subset \hh^1$ be an area-stationary $(X,Y)$-Lipschitz  surface. Let $U$ be an horizontal $C^2$ vector field compactly supported in $\hh^1$, with $\partial S \cap \text{supp}(U)= \emptyset$, and associated flow $\{\varphi_s\}_{s \in \rr}$. Then the second variation of the sub-Finsler area induced by $U$ is given by 
\begin{equation}
\dfrac{d^2}{ds^2}\Big|_{s=0} A_K( \varphi_s(S)) =\int_S \left( Z(f)^2+ q f^2 \right) \dfrac{|N_h|}{\kappa(\pi_K({\nu_h}))} dS,
\end{equation}
where 
\[
q=4\, \bigg\{ Z\bigg( \dfrac{\escpr{N,T}}{|N_h|}\bigg)-  \dfrac{\escpr{N,T}^2}{|N_h|^2}\bigg\},
\]
$\kappa$ is the positive curvature of the boundary $\partial K$ and $f=\escpr{U,\nu_h}$.
\end{theorem}

\begin{remark}
As noticed in the introduction, there is a slight difference in this second variation formula with respect to the sub-Riemannian one computed in \cite{MR3406514}, due to the definition of $Z$ as $-J(\nu_h)$ in this paper, instead of $J(\nu_h)$ as in \cite{MR3406514}. This change was introduced in \cite{2020arXiv200704683P} as the most convenient way of dealing with the lack of symmetry of the sub-Finsler norm.
\end{remark}

\begin{proof}[Proof of Theorem~\ref{th:secondvf}]
First of all we notice that $\escpr{N,T}/|N_h|$ is smooth in the $Z$-direction by~Proposition \ref{pr:satode}, and so $q$ is well defined. Moreover, by Theorem \ref{th:straightfoliation} an area-stationary $(X,Y)$-Lipschitz  surface is $\hh$-regular. The area functional is given by
\begin{align*}
A_K( \varphi_s(S))=\int_S \escpr{\pi(V(s)), V(s)}\, dS,
\end{align*}
where
\begin{equation*}
V(s)=\escpr{E_1(s),T} \,T \times E_2(s) + \escpr{E_2(s),T} \,E_1(s) \times T,
\end{equation*}
and $dS$ is the Riemannian area element. At a regular point $p\in S$, a basis of tangent vectors to  $\varphi_s(S)$ at $\varphi_s(p)$ is given~by $E_1(s)=(d\varphi_s)_p(Z_p)$ and $E_2(s)=(d\varphi_s)_p(E_p)$.

Since by Lemma \ref{lm:novanishing} the norm of $V(s)$ is strictly positive a.e. in $\supp (U) \cap S$, we have 
\begin{equation}
\label{eq:svdomination}
 \begin{aligned}
\dfrac{d^2}{ds^2} \norm{V(s)}_{*}&= \dfrac{d}{ds}\escpr{\pi_K(V(s)), \dfrac{\nabla}{ds} V(s) }\\
&=\escpr{d \pi (V(s)) \dfrac{\nabla}{ds} V(s), \dfrac{\nabla}{ds} V(s) } + \escpr{\pi_K(V(s)), \dfrac{\nabla^2}{ds^2} V(s) } \\
&\le \frac{1}{\bar{\kappa}}\, \big|\dfrac{\nabla}{ds} V(s)\big|^2+ \beta \,\big|\dfrac{\nabla^2}{ds^2} V(s)\big|,
\end{aligned} 
\end{equation}
a.e.~in $\supp (U) \cap S$, where $\bar{\kappa}= \min_{\norm{v}_*=1} \kappa(v)$ and $\beta$ is the positive constant defined in \ref{sc:subper} with $K'$ equal to the Euclidean ball centered at the origin.
Setting $\tilde{N}(s)=E_1(s) \times E_2(s)$ we get
\begin{align*}
\dfrac{\nabla^2}{ds^2} \tilde{N}(s)&=\dfrac{\nabla^2}{ds^2} E_1(s)\times E_2(s)+ 2\dfrac{\nabla}{ds} E_1(s) \times \dfrac{\nabla}{ds} E_2(s)+E_1(s)\times \dfrac{\nabla^2}{ds^2} E_2(s).
\end{align*}
Then Lemmas \ref{lm:cdlb} and \ref{lm:cd2} imply
\begin{align*}
\dfrac{\nabla}{ds} E_i(s)&= \nabla_{E_i(s)} U+2\escpr{J(U_p),E_i(s)}\,T,
\\
  \dfrac{\nabla^2}{ds^2} E_i(s)& =\nabla_{E_i(s)} \nabla_U U +2\, \escpr{J(\nabla_{U} U),E_i(s)} T + 2\, \escpr{J(U_p),\nabla_{E_i(s)} U}\,T.  
\end{align*}
and 
$|E_1(s)|=|d\varphi_s(Z)| \le C'$ and $|E_1(s)|=|d\varphi_s(E)| \le C'$ for $s \in (-s_0,s_0)$, where the constant $C'$ is independent of $s$. Then, writing the covariant derivative $\nabla_{E_i(s)} U$ and $\nabla_{E_i(s)} \nabla_U U$ in standard coordinates, we obtain
\[
\big|\dfrac{\nabla^2}{ds^2} \tilde{N}(s)\big| \le \tilde{C}\, \|U\|_{C^2}
\]
a.e. in $\supp (U) \cap S$ for a suitable constant $\tilde{C}>0$ and where $\|U\|_{C^2}$ denotes the standard $C^2$ norm of $U$.
Then, since $V(s)=\tilde{N}(s)- \escpr{\tilde{N}(s),T}\,T$ and thus $\tfrac{\nabla^2}{ds^2}V(s)=\tfrac{\nabla^2}{ds^2}\tilde{N}(s)- \escpr{\tfrac{\nabla^2}{ds^2}\tilde{N}(s),T}\,T$, we have
\[
\big|\dfrac{\nabla^2}{ds^2} V(s)\big|\le 2\,\big|\dfrac{\nabla^2}{ds^2} \tilde{N}(s)\big|   \le 2\,\tilde{C}\, \|U\|_{C^2}
\]
a.e.~in $\supp (U) \cap S$. 
Then, since $\supp (U) \cap S$ is compact, Lebesgue's dominated convergence theorem yields
\[
\dfrac{d^2}{ds^2}\Big|_{s=0} A_K( \varphi_s(S))=\int_S \dfrac{d^2}{ds^2}\Big|_{s=0} \escpr{\pi(V(s)), V(s)}\, dS.
\]

By Remark 3.3 in \cite{2020arXiv200704683P} we get 
\begin{equation}
\label{eq:UUsv}
\dfrac{d}{ds}  \escpr{\pi(V(s), \dfrac{\nabla}{ds} V(s)}=\escpr{ \dfrac{\nabla}{ds} \pi(V(s), \dfrac{\nabla}{ds} V(s)}+ \escpr{\pi(V(s), \dfrac{\nabla}{ds} \dfrac{\nabla}{ds} V(s)}
\end{equation}
Again by  Remark 3.3 in \cite{2020arXiv200704683P} we have 
\[
0=\dfrac{d}{ds} \escpr{\dfrac{\nabla}{ds} \pi(V(s)), V(s)}=\escpr{\dfrac{\nabla}{ds}\dfrac{\nabla}{ds} \pi(V(s)), V(s) }+ \escpr{\dfrac{\nabla}{ds} \pi(V(s)),\dfrac{\nabla}{ds} V(s) }.
\]
Then we gain 
\begin{equation}
\label{eq:nabnabpi}
\escpr{\dfrac{\nabla}{ds} \pi(V(s)),\dfrac{\nabla}{ds} V(s) }=-\escpr{\dfrac{\nabla}{ds}\dfrac{\nabla}{ds} \pi(V(s)), V(s) }.
\end{equation}
Therefore substituting \eqref{eq:nabnabpi} in \eqref{eq:UUsv} we obtain
\begin{equation*}
\dfrac{d}{ds}  \escpr{\pi(V(s), \dfrac{\nabla}{ds} V(s)}= \escpr{\pi(V(s), \dfrac{\nabla}{ds} \dfrac{\nabla}{ds} V(s)}-\escpr{\dfrac{\nabla}{ds}\dfrac{\nabla}{ds} \pi(V(s)), V(s) }
\end{equation*}
Evaluating the previous equality at $s=0$ we get 
\begin{equation}
    \label{th:secondvf-1}
\dfrac{d^2}{ds^2}\Big|_{s=0} \escpr{\pi(V(s)), V(s)}= \escpr{\pi(\nu_h),\dfrac{\nabla^2}{ds^2}\Big|_{s=0} V(s)}- \escpr{\dfrac{\nabla^2}{ds^2}\Big|_{s=0} \pi(V(s)), \nu_h}=\textbf{I}+\textbf{II},
\end{equation}
since $V(0)= |N_h| \nu_h$.

As 
\begin{align*}
\dfrac{\nabla}{ds} V(s)&=\escpr{\dfrac{\nabla}{ds} E_1(s),T} T \times E_2(s)+ \escpr{E_1(s),T} T \times \dfrac{\nabla}{ds}  E_2(s) \\
& + \escpr{\dfrac{\nabla}{ds} E_2(s),T} E_1(s) \times T+  \escpr{E_2(s),T} \dfrac{\nabla}{ds} E_1(s) \times T.
\end{align*}
we obtain 
\begin{align*}
\dfrac{\nabla^2}{ds^2}\Big|_{s=0} V(s)&=\escpr{\dfrac{\nabla^2}{ds^2}\Big|_{s=0} E_1(s),T} T \times E_2(0)+ 2\escpr{\dfrac{\nabla}{ds}\Big|_{s=0} E_1(s),T} T \times \dfrac{\nabla}{ds}\Big|_{s=0} E_2(s)\\
& + \escpr{E_1(0),T} T \times \dfrac{\nabla^2}{ds^2}\Big|_{s=0} E_2(s) + \escpr{\dfrac{\nabla^2}{ds^2}\Big|_{s=0} E_2 (s),T} E_1(0) \times T \\
&  + 2\escpr{\dfrac{\nabla}{ds}\Big|_{s=0} E_2(s),T} \dfrac{\nabla}{ds}\Big|_{s=0} E_1(s) \times T +\escpr{E_2(0),T} \dfrac{\nabla^2}{ds^2}\Big|_{s=0} E_1 (s) \times T.
\end{align*}
By Lemma \ref{lm:cdlb} 
\[
\dfrac{\nabla}{ds}\Big|_{s=0} E_i (s)= \nabla_{E_i(0)} U + 2 \escpr{J(U),E_i(0)}T,
\]
for $i=1,2$.
By Lemma \ref{lm:cd2} we gain
\[
\dfrac{\nabla^2}{ds^2}\Big|_{s=0} E_i(s)= \nabla_{E_i(0)} ( \nabla_U U)+2 \escpr{J(\nabla_U U),E_i(0)} T + 2 \escpr{J(U),\nabla_{E_i(0)} U}T.\\
\]
Noticing that $\nabla_Z U $ is horizontal we get 
\begin{align*}
\dfrac{\nabla^2}{ds^2}\Big|_{s=0} V(s)&=\left(\escpr{\nabla_{Z} ( \nabla_U U ),T} +2 \escpr{J(\nabla_U U),Z}  + 2 \escpr{J(U),\nabla_{Z} U}\right)  T \times E\\
&\qquad + 4 \escpr{J(U), Z} T \times \nabla_E U\\
&\qquad + \left(\escpr{\nabla_{E} ( \nabla_U U ),T} +2 \escpr{J(\nabla_U U),E}  + 2 \escpr{J(U),\nabla_{E} U}\right) Z \times T \\
&\qquad  + 2 \left( \escpr{\nabla_E U,T} +2 \escpr{J(U),E} \right) \nabla_ Z U \times T -|N_h| \nabla_Z (\nabla_U U) \times T.
\end{align*}
We set $\pi_K(\nu_h)=\pi_Z Z+ \pi_{\nu} \nu_h$, where $\pi_Z=\escpr{\pi(\nu_h), Z}$ and $\pi_{\nu}=\escpr{\pi(\nu_h), \nu_h}$.  Notice that $T\times \nu_h=-Z$, $Z\times T=-\nu_h$ and $\escpr{\pi(\nu_h),W \times T}=\escpr{J(\pi(\nu_h)), W}$ for each vector field $W$, then a straightforward computation shows that 
\begin{equation}
    \label{eq:I=A+B}
\textbf{I}=\escpr{\pi(\nu_h),\dfrac{\nabla^2}{ds^2}\Big|_{s=0} V(s)}=\mathbf{A}+\mathbf{B}
\end{equation}
where
\begin{equation}
\label{eq:A}
\begin{aligned}
    \mathbf{A}&=-2 \escpr{N,T} \escpr{J(U),\nabla_{Z} U} \pi_Z-4 \escpr{J(U), Z} \escpr{J(\pi(\nu_h)), \nabla_E U}\\
&-2 \pi_{\nu} \escpr{J(U),\nabla_E U}+4\escpr{J(\pi(\nu_h)), \nabla_Z U}  \escpr{J(U),E}
\end{aligned}
\end{equation}
and 
\begin{equation}
    \label{eq:B}
\begin{aligned}
    \mathbf{B}&=-\escpr{N,T} Z(\escpr{\nabla_U U,T}) \pi_Z -E(\escpr{\nabla_U U,T}) \pi_{\nu}\\
& \quad-  2\escpr{N,T} \escpr{J(\nabla_U U),\pi(\nu_h)} - |N_h| \escpr{J(\pi(\nu_h)), \nabla_Z \nabla_U U}
\end{aligned}
\end{equation}
Since $S$ is area-stationary, by equation \eqref{eq:first0} in Proposition \ref{prop:fvf}  we have that 
\begin{equation}
\label{eq:1vfin2vf}
\begin{aligned}
    0=\dfrac{d}{ds}\Big|_{s=0} A_K( \varphi_s(S))&=\int_S \bigg[-\escpr{N,T} Z(\escpr{U,T}) \pi_Z -E(\escpr{U,T}) \pi_{\nu}\\
& \quad-  2\escpr{N,T} \escpr{J(U),\pi(\nu_h)} - |N_h| \escpr{J(\pi(\nu_h)), \nabla_Z U}\bigg] dS,\\
\end{aligned}
\end{equation}
for every  $U$  compactly supported $C^1$ vector field. Taking into account the first variation formula \eqref{eq:1vfin2vf} induced by the $C^1$ horizontal vector field  $\nabla_U U$ we get 
\[
\int_S \mathbf{B}=0
\]
Thus we obtain 
\begin{equation}
    \label{eq:intI}
\int_S \textbf{I}=\int_S \mathbf{A}.
\end{equation}

On the other hand, by Lemma \ref{rk:sdpi} and by equation \eqref{eq:nablapi} in Remark \ref{rk:nablapi}, we obtain 
 \[
\textbf{II}=-\escpr{\dfrac{\nabla^2}{ds^2}\Big|_{s=0} \pi(V(s)), \nu_h}= \escpr{\dfrac{\nabla}{ds}\Big|_{s=0} \left( \tfrac{V(s)}{|V(s)|} \right) , (d\pi)_{\nu_h}^* Z} \escpr{ \dfrac{\nabla}{ds}\Big|_{s=0}  V(s) , Z}.
 \]
Then \cite[Lemma 4.3]{MR4314055} yields 
\[
\textbf{II}=\dfrac{ 1}{\kappa \, |N_h|} \escpr{ \dfrac{\nabla}{ds}\Big|_{s=0}  V(s) , Z}^2,
\]
where  $\kappa=\kappa(\pi_K(\nu_h))$ is the positive constant curvature of $\partial K$ evaluated at $\pi_K(\nu_h)$, that is constant along the characteristic curves by Theorem  \ref{th:straightfoliation}. Thanks to \eqref{eq:nUV} we have
 \begin{align*}
 \escpr{ \dfrac{\nabla}{ds}\Big|_{s=0}  V(s) , Z}&=- 2\escpr{N,T}\escpr{J(U),Z}-|N_h| \escpr{\nabla_Z U,\nu_h}\\
  &= 2\escpr{N,T}\escpr{U,\nu_h}-|N_h| \escpr{\nabla_Z U,\nu_h}.
 \end{align*}
 Setting $f= \escpr{U, \nu_h}$ and $g=\escpr{U,Z}$ we get that $\mathbf{A}+\textbf{II}$ is equal to
 \begin{align*}
& 2\escpr{N,T} \pi_Z( g Z(f) +f Z(g))+ 4 f \pi_Z(E(f)-g\theta(E))-2 f\pi_{\nu}(E(g)+f\theta(E))\\
&\qquad -2g \pi_{\nu}(E(f)-g \theta(E)) -4 g \escpr{N,T} \pi_{\nu} Z(g)\\
&\qquad+ \dfrac{1}{\kappa \, |N_h|} (2\escpr{N,T}f -|N_h| Z(f))^2.
 \end{align*}
Notice that
\begin{equation} 
\label{eq:svpt}
\begin{aligned}
\textbf{II}&=\dfrac{1}{  \kappa \, |N_h|} (2\escpr{N,T}f -|N_h| Z(f))^2\\
&= \dfrac{1}{ \kappa \, |N_h|}(4\escpr{N,T}^2 f^2 - 4 \escpr{N,T} |N_h| f Z(f)+|N_h|^2 Z(f)^2).
\end{aligned}
\end{equation}
The second term of  \eqref{eq:svpt} can be written as
\begin{align*}
4 \dfrac{\escpr{N,T} |N_h|}{\kappa \, |N_h| } f Z(f)&=2 |N_h| \dfrac{\escpr{N,T} }{\kappa \, |N_h|} Z(f^2)\\
&=2 |N_h| \left(Z \left( \dfrac{\escpr{N,T} }{\kappa \, |N_h|}f^2\right)- Z\left(\dfrac{\escpr{N,T}}{\kappa \, |N_h|}\right) f^2 \right).
\end{align*}
 Then, setting $h=\tfrac{\escpr{N,T}f^2}{\kappa |N_h|}$ in Lemma \ref{rk:integration by parts}, we obtain that the integrals of the first and second term in \eqref{eq:svpt}  are equal to
 \[
 \int_S  \left(4\,\dfrac{\escpr{N,T}^2}{\kappa |N_h|^2} f^2-2 \dfrac{\escpr{N,T} }{\kappa \, |N_h|} Z(f^2) \right) |N_h| \, dS=2 \int_S Z\left(\dfrac{\escpr{N,T}}{\kappa \, |N_h|}\right) f^2  \ |N_h| \, dS.
 \]
 The integral of the third summand in  \eqref{eq:svpt} is equal to 
 \[
 \int_S Z(f)^2 \, \dfrac{|N_h|}{\kappa} \, dS.
 \]
 Hence we obtain 
 \begin{equation}
     \label{eq:intII}
 \int_S \textbf{II} \, dS=\int_S Z(f)^2 \, \dfrac{|N_h|}{\kappa} \, dS+  2 \int_S Z\left(\dfrac{\escpr{N,T}}{\kappa \, |N_h|}\right) f^2  \ |N_h| \, dS.
 \end{equation}
 
Finally we deal with 
\begin{align*}
\mathbf{A}&= 2\escpr{N,T} \pi_Z( g Z(f) +f Z(g))+ 4 f \pi_Z(E(f)-g\theta(E))-2 f\pi_{\nu}(E(g)+f\theta(E))\\
& -2g \pi_{\nu}(E(f)-g \theta(E)) -4 g \escpr{N,T} \pi_{\nu} Z(g)
\end{align*}
By equation \eqref{eq:thetaE} in Lemma \ref{rk:integration by parts} we have that
\begin{align*}
&2\pi_{\nu}(g^2 \theta(E)-2 \escpr{N,T} g Z(g))\\
&\quad=2\pi_{\nu} |N_h| \left( g^2(2 \dfrac{\escpr{N,T}^2}{|N_h|^2} -Z(\dfrac{\escpr{N,T}}{|N_h|}))-\dfrac{\escpr{N,T}}{|N_h|}Z(g^2)\right)\\
&\quad=2 \norm{N_h}_* \left(2 g^2 \dfrac{\escpr{N,T}^2}{|N_h|^2}-Z\left(\dfrac{\escpr{N,T}}{|N_h|}g^2\right)  \right)
\end{align*}
a.e.~in $S$. Then,
by Lemma \ref{rk:piZN} the integral over $S$ of the previous term is equal to zero.
Therefore we have 
\begin{align*}
\mathbf{A}= 2\escpr{N,T} \pi_Z Z(gf) + 4 f \pi_Z(E(f)-g\theta(E))-2 f^2 \pi_{\nu}\theta(E)  -2 \pi_{\nu} E(gf)
\end{align*}
On the one hand we have 
\[
\int_{S} \pi_{\nu} E(gf)+ \pi_Z \theta (E) gf =0,
\]
by equation \eqref{eq:intSE} in Lemma \ref{rk:piZN}. On the other hand using equation \eqref{eq:intbyparts2} in Lemma \ref{rk:piZN} and equation \eqref{eq:thetaE} in Lemma \ref{rk:integration by parts}  we have 
\begin{align*}
&2\int_{S} \pi_{Z} \left(\escpr{N,T} Z(gf)-gf \theta(E) \right)=\\
&2\int_{S}( Z\left(\dfrac{\escpr{N,T}}{|N_h|}  \pi_{Z} gf \right)- 2\dfrac{\escpr{N,T}^2}{|N_h|^2}  \pi_{Z} gf )|N_h| dS=0.
\end{align*}
Hence we gain 
\begin{equation}
    \label{eq:A-2}
\begin{aligned}
\mathbf{A}&= 2 \pi_Z E(f^2) -2 f^2 \pi_{\nu}\theta(E)=-2 f^2(E(\pi_Z)+\pi_{\nu} \theta(E))\\
&=-2 f^2(\escpr{\nabla_E \pi(\nu_h),Z}- \pi_{\nu} \theta(E)+\pi_{\nu} \theta(E))\\
&=-2f^2 \escpr{(d\pi)_{\nu_h} ( \nabla_E \nu_h), Z}=-2f^2 \dfrac{\theta(E)}{\kappa(\pi(\nu_h))},
\end{aligned}
\end{equation}
a.e.~in $S$, where the last inequality follows from Lemma~4.3 in  \cite{MR4314055}. Finally, by \eqref{eq:thetaE} in Lemma \ref{rk:integration by parts} and \eqref{eq:intI}, \eqref{eq:intII} and \eqref{eq:A-2} we obtain
\begin{equation}
\label{eq:svf}
\int_S \textbf{I}+\textbf{II}= \int_S \mathbf{A}+\textbf{II}=\int_S \left( Z(f)^2+4 \left( Z\left( \dfrac{\escpr{N,T}}{|N_h|}\right)-  \dfrac{\escpr{N,T}^2}{|N_h|^2}\right) f^2 \right) \dfrac{|N_h|}{\kappa(\pi({\nu_h}))} dS.
\end{equation}
In the last equation we use that $Z(\pi(\nu_h))=0$, therefore $\kappa(\pi({\nu_h}))$ is constant along the characteristic curves.
\end{proof}

\begin{lemma}
\label{rk:sdpi}
Following the previous notation we have
\[
 \escpr{\dfrac{\nabla}{ds} \dfrac{\nabla}{ds}\Big|_{s=0} \pi(V(s)), V(s)}=-\escpr{\dfrac{\nabla}{ds}\Big|_{s=0} \pi(V(s)), Z} \escpr{ \dfrac{\nabla}{ds}\Big|_{s=0} V(s), Z}.
\]
\end{lemma}
\begin{proof}
By Remark~3.3 in \cite{2020arXiv200704683P} we have  $\escpr{\dfrac{\nabla}{ds} \pi(V(s)), V(s)}=0$. Then 
\[
\dfrac{\nabla}{ds} \pi(V(s))=f(s) J(\tfrac{V(s)}{|V(s)|}),
\]
where $f(s)=\escpr{\dfrac{\nabla}{ds} \pi(V(s)), J(\tfrac{V(s)}{|V(s)|})}$. Since 
\[
 \dfrac{\nabla}{ds} \dfrac{\nabla}{ds} \pi(V(s))= \dfrac{d}{ds} f(s) J(\tfrac{V(s)}{|V(s)|})+ f(s) \dfrac{\nabla}{ds} J(\tfrac{V(s)}{|V(s)|})
\]
and $\nabla J=0$ we obtain 
\[
 \escpr{\dfrac{\nabla}{ds} \dfrac{\nabla}{ds} \pi(V(s)), V(s)}= f(s)\escpr{ \dfrac{\nabla}{ds} J(\tfrac{V(s)}{|V(s)|}), V(s)}=- f(s) \escpr{ \dfrac{\nabla}{ds} \left( \tfrac{V(s)}{|V(s)|} \right) , J(V(s))}.
\]
Evaluating  at $s=0$ we gain
\[
 \escpr{\dfrac{\nabla}{ds} \dfrac{\nabla}{ds}\Big|_{s=0} \pi(V(s)), V(s)}=-\escpr{\dfrac{\nabla}{ds}\Big|_{s=0} \pi(V(s)), Z} \escpr{ \dfrac{\nabla}{ds}\Big|_{s=0} V(s), Z},
\]
since $V(0)= |N_h| \nu_h$.
\end{proof}

\begin{remark}
\label{rk:nablapi}
Letting
\[
\pi(V(s))= \pi_1(V(s)) X_{\ga(s)} + \pi_2(V(s)) Y_{\ga(s)},
\]
and noticing that $\nabla X=\nabla Y=0$
we gain
\[
\dfrac{\nabla}{ds}\Big|_{s=0} \pi(V(s))= \dfrac{d}{ ds}\Big|_{s=0} \pi_1(V(s)) X_{\ga(0)} + \dfrac{d}{ ds}\Big|_{s=0} \pi_2(V(s)) Y_{\ga(0)}.
\]
Setting $\nu_h=a X+ b Y$ 
 we obtain 
\begin{equation}
\label{eq:nablapi}
\dfrac{\nabla}{ds}\Big|_{s=0} \pi(V(s))= (d \pi)_{(a,b)} \left(\dfrac{\nabla}{ds}\Big|_{s=0} \tfrac{V(s)}{|V(s)|}  \right)    
\end{equation}
where 
\[
 (d \pi)_{(a,b)}=\begin{pmatrix} \dfrac{\partial \pi_1 }{\partial a}(a,b) & \dfrac{\partial \pi_1}{\partial b} (a,b)\\[3mm] \dfrac{\partial \pi_2}{\partial a} (a,b) & \dfrac{\partial \pi_2}{\partial b} (a,b) \end{pmatrix}.
\]
By Corollary 1.7.3 in \cite{MR3155183} $\pi_K= \nabla h$ where $h$ is a $C^2$ function, thus by Schwarz's theorem  the Hessian $\text{Hess}_{(a,b)}(h)=(d \pi)_{(a,b)}$ is symmetric.
\end{remark}

\begin{lemma}\label{rk:integration by parts} Let $S \subset \hh^1$ be a $(X,Y)$- Lipschitz  surface. Let $h$ be a compactly supported function on $S$, differentiable in the $Z$-direction. Then we have 
\begin{equation}
\label{eq:intbyparts}
\int_{S} \left( Z(h) - 2 \dfrac{\escpr{N,T}}{|N_h|} h \right) |N_h| dS=0
\end{equation}
and 
\begin{equation}
\label{eq:thetaE}
\theta(E)=-|N_h| Z\left( \dfrac{\escpr{N,T}}{|N_h|}\right)+ 2|N_h| \dfrac{\escpr{N,T}^2}{|N_h|^2}
\end{equation}
a.e. in $S$, where $\theta(E)=\escpr{\nabla_E \nu_h, Z}$.
\end{lemma}
\begin{proof}
Following Proposition~1.20 in \cite[]{MR657581} or Remark~6.1 \cite{MR3044134}  we approximate the $(X,Y)$-Lipschitz surface $S=\{p \in \hh^1 : f(p)=0\}$ by a family of smooth surfaces  $S_j=\{p \in \hh^1  :  f_j(p)=0\}$, where $f_j=\rho_j * f$ and  $\rho_j$ are the standard Friedrichs’ mollifiers, that converges to $S$ on compact subsets of $S$. Let $Z^j$, $N^j$ and $E^j$ relative to $S_j$. Then we have
\[
\divv(|N_h^j| h Z^j)=|N^j_h| Z^j(h)-\escpr{N^j,T} h \left( -\dfrac{Z^j(|N_h^j|) }{\escpr{N^j,T}} + |N_h^j| \theta(E^j) + 2 |N_h^j|^2\right).
\]
Using  $-|N_h^j|^{-1}Z^j(\escpr{N^j,T})=\escpr{N^j,T}^{-1} Z^j(|N_h^j|)$, $|N_h^j|^{-1}Z^j(\escpr{N^j,T})- 2\escpr{N^j,T}^2+ |N^j_h| \theta(E^j)=0$ and the divergence theorem  we get
\[
\int_{S_j}  |N_h^j| Z^j(h) - 2 {\escpr{N^j,T}} h   \, dS_j=0.
\]
Then, passing to the limit when $j\to +\infty$ we obtain \eqref{eq:intbyparts}. Since $S_j$ are smooth a straightforward computation shows that
\[
\theta(E^j)=-|N_h^j| Z^j\left( \dfrac{\escpr{N^j,T}}{|N^j_h|}\right)+ 2|N^j_h| \dfrac{\escpr{N^j,T}^2}{|N^j_h|^2}.
\]
Passing to the limit when $j\to +\infty$ we have $E^j \to E$ a.e.~in $S$, since $S$ is Euclidean Lipschitz. Therefore we obtain that \eqref{eq:thetaE} holds a.e. in $S$.
\end{proof}

\begin{lemma}\label{rk:piZN}
Let $S \subset \hh^1$ be an area-stationary $(X,Y)$-Lipschitz surface. Let $h$ be a compactly supported function in $S$, differentiable in the $Z$ direction then we have 
\begin{equation}
\label{eq:intbyparts2}
\int_{S} \left( Z(h) - 2 \dfrac{\escpr{N,T}}{|N_h|} h \right) \norm{N_h}_* dS=0.
\end{equation}
Moreover, there holds 
\begin{equation}
\label{eq:intSE}
\int_S \pi_{\nu} E(h)+\pi_Z \theta(E) h=0.
\end{equation}
\end{lemma}
\begin{proof}
Let $\pi_{\nu}=\escpr{\pi(\nu_h),\nu_h}=\norm{\nu_h}_*$, then $\norm{N_h}_*=|N_h|\pi_{\nu}$. Since $S$ is an area-stationary surface,  Theorem \ref{th:straightfoliation} implies that $\nu_h$ is constant in the $Z$ direction, thus in particular we have 
$
Z(\pi_{\nu})=0.
$
Therefore, applying the same divergent argument of the proof of Lemma \ref{rk:integration by parts} we obtain
\[
\int_{S} \left( Z(h) - 2 \dfrac{\escpr{N,T}}{|N_h|} h \right) \norm{N_h}_* dS=0.
\]
Always following \cite[Proposition 1.20]{MR657581} or  \cite[Remark 6.1]{MR3044134}  we approximate the $(X,Y)$-Lipschitz surface $S=\{p \in \hh^1  : f(p)=0\}$ by a family of smooth surfaces  $S_j=\{p \in \hh^1  : f_j(p)=0\}$, where $f_j=\rho_j * f$ and  $\rho_j$ are the standard Friedrichs’ mollifiers, that converges to $S$ on compact subsets of $S $. Let $Z^j$, $N^j$ and $E^j$ relative to $S_j$.  Using \cite[Remark 3.3]{2020arXiv200704683P} it is easy to prove that $E(\pi_{\nu^j})=\pi_{Z^j} \theta(E^j)$. Thus,  by Proposition \ref{pr:divt} we gain  
\begin{equation}
\label{eq:intbypartE}
\int_{S_j} \pi_{\nu^j} E^j(h)+ \pi_{Z^j} \theta(E^j) h \ dS_j=-\int_{S_j} \escpr{N^j,T} \theta(Z^j) h \ dS_j.
\end{equation}
Since $S$ is area-stationary, we get $H_K=0$ and Proposition 4.2 in \cite{MR4314055} implies $H_D=\escpr{\nabla_{Z} \nu_h, Z}=0$.
Then $\theta(Z^j)= \escpr{\nabla_{Z^j} \nu_h^j, Z^j}\to \escpr{\nabla_{Z} \nu_h, Z}=0$ and, passing to the limit in \eqref{eq:intbypartE} when $j\to +\infty$, we obtain \eqref{eq:intSE}.
\end{proof}

\section{The Bernstein's problem for $(X,Y)$-Lipschitz surfaces}
\label{sec:bernstein}

\begin{definition}
We say that a complete oriented area-stationary $(X,Y)$-Lipschitz surface $S \subset \hh^1$ is stable if inequality
\begin{equation}
\label{eq:stable}
\int_S \left( Z(f)^2+4 \left( Z\left( \dfrac{\escpr{N,T}}{|N_h|}\right)-  \dfrac{\escpr{N,T}^2}{|N_h|^2}\right) f^2 \right) \dfrac{|N_h|}{\kappa(\pi({\nu_h}))} dS \ge0
\end{equation}
holds for any continuous function $f$ on $S$ with compact support such that $Z(f)$ exists and is continuous.
\end{definition}

The following lemma is proven in \cite[page 45]{MR2333095}.

\begin{lemma}
\label{lm:ASCV}
Let $A,B \in \rr$ be such that $A^2\le2B$ and set $h(s):=1+As +B s^2 /2$. If
\[
\int_{\rr} \phi'(s)^2 h(s) ds \ge (2B-A^2) \int_{\rr} \phi(s)^2 \dfrac{1}{h(s)} ds
\]
for each $\phi \in C_0^1(\rr)$ then $2B=A^2$.
\end{lemma}

\begin{theorem}[Bernstein's theorem]
\label{thm:bernstein} Let $K\in C^2_+$ be a convex body with $0 \in \intt(K)$.
Let $S \subset \hh^1$ be a complete, connected and stable $(X,Y)$-Lipschitz surface. Then $S$ is a vertical plane.
\end{theorem}

\begin{proof}
First of all we have that $S$ is an $\hh$-regular surface by Theorem \ref{th:straightfoliation}. 
Let $p$ in $S$. Since $S$ is $(X,Y)$-Lipschitz, by Theorem \ref{th:iftLip}, there exist an open ball $B_r(p)$ and a Lipschitz function $u:D \to \rr$ such that $S \cap B_r(p)=\text{Gr}(u)$ where $\text{Gr}(u)=\{(x,u(x,y),t-xu(x,t))\in \hh^1 \ : \ (x,t) \in D\}$.  
 Let $(0,0) \in D$ be the projection of $p$ to the $xt$-plane. On $D$ we consider the coordinates around $(0,0)$ furnished by $G(s,\eps)$ defined in Lemma \ref{lm:teps>0}. Let $I$ be a small interval containing 0, then $\eps \in I $ and $s \in ]-r, r[$. Since $S$ is complete by the Hopf-Rinow Theorem each geodesic (in particular the straight lines in the $Z$-direction) can be indefinitely extended along any direction, thus the open interval $]-r,r[$ extend  to $\rr$.
Notice that $\bar{\ga}_{\eps}(s)$ is the integral curve of $Z$, thus $Z(f)=\partial_s(f)$. Hence, taking into account that $(u_x+2uu_t)(s)$ is constant along $\bar{\ga}_\eps$ and equal to $b(\eps)$, the stability condition \eqref{eq:stable} is equivalent to 
\begin{equation}
\label{eq:stability2}
\int_{I} \int_{\rr} \bigg( (\partial_s f)^2- 4\bigg(\dfrac{\escpr{N,T}^2}{|N_h|^2} - \partial_s \bigg(\dfrac{\escpr{N,T}}{|N_h|}\bigg) \bigg) f^2\bigg)\dfrac{ \partial t_{\eps}}{\partial \eps} \dfrac{\sqrt{1+b(\eps)^2}}{\kappa(\pi({\nu_h}))} \, ds \,d\eps \ge 0,
\end{equation}
for any continuous function $f$ on $S$ with compact support such that $Z(f)$ exists and is continuous.

Since $\escpr{N,T}/{|N_h|}$ solves the equation \eqref{eq:ODECodazzi} with initial condition $y(0)=a'(\eps)/2$ and $y'(0)=a'(\eps)^2/2-b'(\eps)$, by \eqref{eq:uabq} we get 
\[
\dfrac{\escpr{N,T}^2}{|N_h|^2}- \left(\dfrac{\escpr{N,T}}{|N_h|}\right)'=\dfrac{b'(\eps)-\tfrac{a'(\eps)^2}{4}}{(1-a'(\eps)s+b'(\eps)s^2)^2}.
\]
Therefore, computing $\ptl t_\eps/\ptl\eps$ from \eqref{eq:taylorex}, we obtain that \eqref{eq:stability2} is equivalent to
\[
\int_{I} \int_{\rr} \bigg( (1-a'(\eps)s+b'(\eps)s^2) (\partial_s f)^2- \dfrac{4b'(\eps)-a'(\eps)^2}{(1-a'(\eps)s+b'(\eps)s^2)} f^2 \bigg) \dfrac{\sqrt{1+b(\eps)^2}}{\kappa(\pi({\nu_h}))} \, ds\, d\eps \ge 0.
\]

Let $\eta: \rr \to \rr$ be a positive function compactly supported in $\rr$ and for $\rho>0$ we consider the family $\eta_{\rho}(x)=\rho^{-1} \eta(x/\rho)$, that weakly converge to the Dirac delta distribution. Putting  the test functions $\eta_{\rho}(x-\eps) \psi(s)$, where $\psi \in C^1_0(\rr)$, in  the previous equation and letting $\rho\to 0$ we get
\[
\int_{\rr}  (1-a'(\eps)s+b'(\eps)s^2) (\psi'(s))^2 ds \ge (4b'(\eps)-a'(\eps)^2)  \int_{\rr} \dfrac{\psi(s)^2}{(1-a'(\eps)s+b'(\eps)s^2)}\,   ds,
\]
for a.e. $\eps$ since  $\kappa(\pi(\nu_h))$ is a positive constant along the horizontal straight lines for each $\eps$ (since $\nu_h$ is constant along such horizontal straight lines) and $\sqrt{1+b(\eps)^2}$ is a positive constant on $\bar{\ga}_\eps$.

Setting $A=-a'(\eps)$, $B=2 b'(\eps)$ and $h(s):=1+As +B s^2/2$, we obtain 
\[
\int_{\rr} h(s) \psi'(s)^2 ds \ge (2B-A^2) \int_{\rr} \dfrac{\psi^2(s)}{h(s)} ds
\]
for each $\psi \in C^1_0 (\rr)$. Assume that $2B-A^2 \ge0$ then by Lemma \ref{lm:ASCV} we get that $2B=A^2$, then $4b'(\eps)-a'(\eps)^2=0$. Therefore by Lemma \ref{lm:ODECodazzi}  we obtain $\escpr{N,T} \equiv 0$, $a'(\eps)=b'(\eps)=0$ a.e. in $\eps$. On the other hand, if $2B-A^2 <0$ then directly by Lemma \ref{lm:ODECodazzi}  we obtain $\escpr{N,T} \equiv 0$, $a'(\eps)=b'(\eps)=0$ a.e. in $\eps$. Hence  $a(\eps)$ and $b(\eps)$ are constant functions in $\eps$ and 
\[
t_{\eps}(s)=\eps+a s+ b s^2,
\]
for some constant $a,b \in \rr$.
Since $t'_{\eps}(s)=2u(s,t_{\eps}(s))=2 \tilde{u}(s,\eps)$ we get 
$
\tilde{u}(s,\eps)= a/2 + b s
$, thus $\tilde{u}$ is an affine function.  Hence $S$ is locally a strip contained  in a vertical plane. A standard connectedness argument  implies that each connected component of $S$ is a vertical plane.
\end{proof}

\bibliographystyle{abbrv} 
\bibliography{sub-finsler}

\end{document}